\documentclass[11pt, twoside]{article}

\usepackage[english]{babel}

\usepackage{bbm}
\usepackage{amssymb}
\usepackage{amsfonts}
\usepackage{amsmath}
\usepackage{amsthm}
\usepackage{color}
\usepackage{mathrsfs}
\usepackage{txfonts}
\usepackage{bbm}
\usepackage{enumerate}
\usepackage{anysize}
\usepackage{indentfirst}

\usepackage[T1]{fontenc}
\usepackage{latexsym}

\usepackage[colorlinks=true,
linkcolor=blue,
citecolor=red,
urlcolor=magenta,
]{hyperref}

\allowdisplaybreaks

\newtheorem{theorem}{Theorem}[section]
\newtheorem{letteredtheorem}{Theorem}

\newtheorem{lemma}[theorem]{Lemma}
\newtheorem{corollary}[theorem]{Corollary}
\newtheorem{proposition}[theorem]{Proposition}

\theoremstyle{definition}
\newtheorem{remark}[theorem]{Remark}

\newcounter{assum}

\renewcommand{\appendix}{\par
\setcounter{section}{0}%
\setcounter{subsection}{0}%
\setcounter{subsubsection}{0}%
\gdef\thesection{\@Alph\c@section}%
\gdef\thesubsection{\@Alph\c@section.\@arabic\c@subsection}%
\gdef\theHsection{\@Alph\c@section.}%
\gdef\theHsubsection{\@Alph\c@section.\@arabic\c@subsection}%
\csname appendixmore\endcsname
}

\numberwithin{equation}{section}

\begin{document}

\arraycolsep=1pt

\title{\bf\Large
Optimal Weak-Type Estimates and Their Applications
of Lifted Rough Maximal Operators
\footnotetext{\hspace{-0.35cm} 2020
\emph{Mathematics Subject Classification}. Primary 42B25;
Secondary 42B20, 42B30.
\endgraf \emph{Key words and phrases}.
lifted rough maximal operator, generalized Poisson integral,
weak-type boundedness, Hardy space.
\endgraf
This project is partially supported by the National Natural Science Foundation of China
(Grant Nos. 12431006 and 12371093), the Beijing Natural Science
Foundation (Grant No. 1262011), and
the Fundamental Research Funds
for the Central Universities (Grant No. 2253200028).}}
\author{Dachun Yang\footnote{Corresponding author,
E-mail: \texttt{dcyang@bnu.edu.cn}/{\color{red} \today}/Final version.}, \
Wen Yuan and Yirui Zhao
}
\date{}
\maketitle

\vspace{-0.6cm}

\begin{center}
\begin{minipage}{13cm}
{\small {\bf Abstract}\quad
Let $n\in\mathbb N\cap[2,\infty)$ and $\Omega\in L^1(\mathbb S^{n-1})$ with $\Omega\not\equiv 0$.
In this article, we introduce a new family of lifted rough maximal operators
$\{\mathcal{M}_\theta^\Omega\}_{\theta\in(0,\infty)}$ in the upper-half plane and establish their optimal
weak-type estimates.
Specifically, we prove that, for any $p \in (1, \infty)$,
the estimate, with the positive equivalence constants independent of $f$,
\[
\sup_{\theta,\lambda\in(0,\infty)}\lambda^p
\underset{{\mathcal M}^\Omega_\theta(f)(x,t) > \lambda t^\frac{\gamma}{p}}
{\int_{\mathbb R^n}\int_0^\infty}
t^{\gamma-1}\,dt\,dx \sim \|f\|_{L^p(\mathbb{R}^n)}^p
\]
holds for all $f\in L^p(\mathbb R^n)$ if and only if $\gamma\in\mathbb R\setminus\{0\}$.
For the endpoint case $p=1$ and $\Omega \in L(\log L)(\mathbb{S}^{n-1})$,
we prove that the above estimate
holds if and only if $\gamma \in (-\infty, -n) \cup (0, \infty)$.
As applications, we obtain weak-type estimates for generalized
Poisson integrals without any logarithmic integrability assumptions, which
gives an affirmative answer to the question posed by
Sj\"ogren and Soria in page 228 of [Israel J. Math. 95 (1996)].
Moreover, although the operator $M^\ast_\Omega$, arising from the method of rotation of
Calder\'on and Zygmund, is not of weak type $(1,1)$,
we find that its lifted variant is weak type $(1,1)$. In addition,
we establish a new characterization of Hardy spaces in terms of
truncated rough singular integrals.
}
\end{minipage}
\end{center}

\vspace{0.2cm}

\tableofcontents

\vspace{0.2cm}

\section{Introduction}

The study of maximal operators with rough kernels is deeply rooted
in the theory of singular integrals and the boundary value problems of harmonic functions.
Recall that the \emph{rough maximal operator $M^{\Omega}$}
is  defined by setting, for any $f\in L^1_{\mathrm{loc}}(\mathbb{R}^{n})$
and $x\in\mathbb{R}^n$,
\begin{align*}
M^\Omega(f)(x):=\sup_{r\in(0,\infty)}\frac{1}{|B(x,r)|}
\int_{B(x,r)}|f(y)|\left|\Omega\left(\frac{x-y}{|x-y|}\right)\right|dy,
\end{align*}
where $\Omega\in L^1(\mathbb S^{n-1})$ and $L^1_{\mathrm{loc}}(\mathbb{R}^{n})$
denotes the set of all locally integrable functions on $\mathbb R^n$.
It is known that $M^\Omega$ is strong $(p,p)$ by the method of
rotation of Calder\'on and Zygmund \cite{cz56} if $p\in(1,\infty]$. Later on,
Stein \cite{s67} utilized this operator to bound the maximal function arising from the
restricted convergence of Poisson integrals on multi-dimensional tube domains,
and posed the problem whether $M^\Omega$ is of weak type $(1,1)$,
which has become a long-standing open problem in harmonic analysis.

R. Fefferman \cite{f78} solved this problem under the stronger condition that $\Omega$
is of "finite entropy" and this was extended by Soria \cite{s85} to the case where
$\Omega$ is in the (larger) block space $B_\infty$. Subsequently,
Christ \cite{c88} proved this result for $n=2$ and $\Omega\in L^q(\mathbb S^1)$
with $q>1$, which further improved by Christ and Rubio de Francia \cite{cr88}
to arbitrary $n\ge 2$ and $\Omega\in L(\log L)(\mathbb S^{n-1})$.
Recall that a measurable function $\Omega\in L(\log L)(\mathbb S^{n-1})$
if
\[
|\Omega|_{L(\log L)(\mathbb S^{n-1})}
:=
\int_{\mathbb S^{n-1}}
\left|\Omega(\theta)\right|
\left[1+\log^+\left|\Omega(\theta)\right|\right]
\,d\sigma(\theta)<\infty,
\]
where $\log^+ z:=\max\{0,\log z\}$ for any $z\in(0,\infty)$.
We summarize the classical result for $M^\Omega$ as follows.
\begin{letteredtheorem}\label{thm-classical-rough-maximal}
Let $n\in\mathbb{N}\cap [2,\infty)$ and
$\Omega\in L^1(\mathbb{S}^{n-1})$.
Then, for any $p\in(1,\infty]$, there exists a positive constant
$C$, depending only on $n$, $p$, and $\|\Omega\|_{L^1(\mathbb S^{n-1})}$, such that, for any
$f\in L^p(\mathbb{R}^n)$,
\begin{align*}
\left\|M^\Omega(f)\right\|_{L^p(\mathbb{R}^n)}
\le C
\|f\|_{L^p(\mathbb{R}^n)}.
\end{align*}
Moreover, if $\Omega\in L(\log L)(\mathbb{S}^{n-1})$,
then there exists a positive constant $C$, depending only on $n$ and
$|\Omega|_{L(\log L)(\mathbb S^{n-1})}$, such that, for any $f\in
L^1(\mathbb{R}^{n})$,
\begin{align*}
\sup_{\lambda\in(0,\infty)}\lambda\left|\left\{
x\in\mathbb{R}^n:M^\Omega(f)(x)>\lambda\right\}\right|
\le C\|f\|_{L^1(\mathbb{R}^n)}.
\end{align*}
\end{letteredtheorem}
Since then, maximal operators with rough kernels have served as a fundamental object
in the study of rough (maximal) singular integrals,
and their boundedness properties
have played an important role in the subsequent development of harmonic analysis;
see, for example, \cite{st01,c04,dl17,h87,l24,qwxyz26,ss96,ss97}.
For recent advances in rough (maximal) singular integrals, we refer to
\cite{bm23,bs25+,bs26+,chl23,l25}.
For further study of the operators with nonsmooth kernels arising from harmonic analysis,
we refer to \cite{bcdh17,bd20,cdhl21,cdhl18,hltx25,s96}.

When $\Omega\equiv 1$, the operator $M^\Omega$ coincides with
the classical \emph{Hardy--Littlewood maximal operator $M$}, which is
defined by setting, for any $f\in L^1_{\rm loc}(\mathbb R^{n})$ and $x\in\mathbb R^n$,
\begin{align*}
M(f)(x) := \sup_{r \in (0,\infty)} \frac{1}{|B(x,r)|}\int_{B(x,r)} |f(y)|\,dy.
\end{align*}
It is well known that the operator $M$ is strong $(p,p)$ for all $p\in(1,\infty]$
and is of weak type $(1,1)$.
However, in certain specific applications arising from harmonic analysis and
partial differential equations,
the Hardy--Littlewood maximal operator is inherently too coarse
to resolve the fine-scale behavior of the objects we aim to control.
As a finer alternative,
Dai et al.\ \cite{dlyyz26} recently introduced a
family of lifted Hardy--Littlewood maximal operators.
Specifically, given $\theta \in (0,\infty)$, the \emph{lifted
Hardy--Littlewood maximal operator $\mathcal{M}_\theta$} is defined by setting,
for any $f \in L^1_{\mathrm{loc}}(\mathbb{R}^n)$ and
$(x,t) \in \mathbb R^n\times (0,\infty)=:\mathbb{R}^{n+1}_+$,
\begin{align}\label{eq-lifted-hl-def}
\mathcal{M}_\theta(f)(x,t) := \sup_{r \in [\theta t, 2\theta t)}
\frac{1}{|B(x,r)|}\int_{B(x,r)} |f(y)|\,dy.
\end{align}
In \cite[Theorems 2.3 and 2.6]{dlyyz26}, Dai et al. proved that the
weak-type estimate
\begin{align}\label{eq-lifted-hl-weak}
\sup_{\theta,\lambda\in(0,\infty)}\lambda^p
\underset{{\mathcal M}_\theta(f)(x,t) > \lambda t^{\frac{\gamma}{p}}}
{\int_{\mathbb R^n}\int_0^\infty} t^{\gamma-1}\,dt\,dx \sim
\|f\|_{L^p(\mathbb{R}^n)}^p
\end{align}
holds for all $f\in L^p(\mathbb R^n)$
if and only if $p\in(1,\infty)$ and $\gamma\in\mathbb{R}\setminus\{0\}$
or $p=1$ and $\gamma \in (-\infty, -n) \cup (0, \infty)$,
where the positive equivalence constants are independent of $f$.
Here, and thereafter, to simplify the presentation, we always use
the following abbreviation:
\begin{align*}
\underset{g(x,t)>h(x,t)}{\int_{\mathbb R^n}\int_0^\infty}
\cdots\,t^{\gamma-1}\,dt\,dx
:=\int_{\mathbb R^n}\left[
\int_{\{t\in(0,\infty):\,g(x,t)>h(x,t)\}}
\cdots\,t^{\gamma-1}\,dt
\right]dx.
\end{align*}

Clearly, for any given $\theta,t\in(0,\infty)$ and $f\in L^1_{\rm loc}(\mathbb R^n)$
and for any $x\in\mathbb{R}^n$,
$\mathcal{M}_\theta(f)(x,t)\le M(f)(x)$.
Therefore, the multi-scale and the lifted structures of
the family $\{\mathcal{M}_\theta\}_{\theta\in(0,\infty)}$
make them capture those information that is lost due
to the excessive magnitude of the operator $M$ and
have been applied to refine some basic estimates related to the operator $M$ in harmonic analysis
in \cite{dlyyz26}. The basic idea here is to
replace the classical Hardy--Littlewood maximal
operator $M$ by a family of the lifted Hardy--Littlewood
maximal operators $\{\mathcal{M}_\theta\}_{\theta\in(0,\infty)}$
that enjoy local fine properties,
thereby overcoming the coarseness mentioned
above. For instance,
via the refined Lusin--Lipschitz inequalities,
\eqref{eq-lifted-hl-weak} is closly
connected with the recent works of Brezis et al. \cite{bvy21,bsvy24}, concerning the weak-type
characterization of Sobolev spaces via difference quotients.
Combined with the refined Cotlar inequality,
\eqref{eq-lifted-hl-weak} is also applied to establish new
characterizations of Hardy(--Sobolev) spaces in terms of
truncated Riesz transforms.
Moreover, \eqref{eq-lifted-hl-weak} provides a framework for the endpoint
$p=1$ within the setting of the one-parameter family of operators studied
by Dom\'inguez and Milman \cite{dm22}.

Motivated by \eqref{eq-lifted-hl-weak}, it is
natural to ask whether the weak-type estimates of lifted
maximal operators with rough kernels holds.
In this article, we affirmatively answer this question by introducing a lifted version of the rough
maximal operator. Precisely speaking,
given $\Omega\in L^1(\mathbb S^{n-1})$ and $\theta\in(0,\infty)$,
the \emph{lifted rough maximal operator ${\mathcal M}_\theta^\Omega$}
is defined by setting,
for any $f\in L^1_{\mathrm{loc}}(\mathbb{R}^n)$ and $(x,t)\in\mathbb{R}^{n+1}_+$,
\[
{\mathcal M}_\theta^\Omega(f)(x,t):=
\sup_{r\in[\theta t,2\theta t)}\frac{1}{|B(x,r)|}\int_{B(x,r)}|f(y)|
\left|\Omega\left(\frac{x-y}{|x-y|}\right)\right|dy.
\]
Moreover, we prove the following optimal weak-type estimates for
${\mathcal M}_\theta^\Omega$.

\begin{theorem}\label{thm-diag}
Let $n\in\mathbb{N}\cap[2,\infty)$, $\gamma\in\mathbb{R}$, and
$\Omega\in L^1(\mathbb S^{n-1})$ with $\Omega\not\equiv0$. Then the following
statements hold.
\begin{enumerate}[{\rm (i)}]
\item If $p\in(1,\infty)$, then the estimate
\begin{equation}\label{eq-p>1}
\sup_{\theta,\lambda\in(0,\infty)}\lambda^p
\underset{{\mathcal M}^\Omega_\theta(f)(x,t) > \lambda t^\frac{\gamma}{p}}
{\int_{\mathbb R^n}\int_0^\infty}
t^{\gamma-1}\,dt\,dx \sim \|f\|_{L^p(\mathbb{R}^n)}^p,
\end{equation}
holds for any $f \in L^p(\mathbb{R}^n)$ if and only if
\(\gamma\in\mathbb R\setminus\{0\}\), where the positive equivalence
constants are independent of $f$.
\item If further assume that $\Omega \in L(\log L)(\mathbb S^{n-1})$,
then the estimate
\begin{equation}\label{eq-p=1}
\sup_{\theta,\lambda\in(0,\infty)}\lambda
\underset{{\mathcal M}^\Omega_\theta(f)(x,t) > \lambda t^\gamma}
{\int_{\mathbb R^n}\int_0^\infty}
t^{\gamma-1}\,dt\,dx \sim \|f\|_{L^1(\mathbb{R}^n)},
\end{equation}
holds for any $f \in L^1(\mathbb{R}^n)$ if and only if
\(\gamma\in(-\infty,-n)\cup(0,\infty)\),
where the positive equivalence
constants are independent of $f$.
\end{enumerate}
\end{theorem}
\begin{remark}
Theorem \ref{thm-diag} when $\Omega\equiv1$
coincides with \eqref{eq-lifted-hl-weak}.
To the best of our knowledge,
Theorem \ref{thm-diag} in other cases are new.
\end{remark}

As in the proof of \eqref{eq-lifted-hl-weak}, the most involved part of
the proof of Theorem \ref{thm-diag} is the one of \eqref{eq-p=1} for $p=1$
and $\gamma\in(-\infty,-n)\cup(0,\infty)$.
However, the ideas behind the proof of \eqref{eq-p=1}
are essentially different from the ones of \eqref{eq-lifted-hl-weak}.
Indeed,
to estimate the size of the level sets of the operator $\mathcal{M}_\theta$,
the proof of \eqref{eq-lifted-hl-weak} in \cite{dlyyz26}
relies on a "global" geometric covering lemma because the kernel of $\mathcal{M}_\theta$
is isotropic. However, this isotropy breaks down in
the operator $\mathcal{M}_\theta^\Omega$ with a rough kernel $\Omega\in L^1(\mathbb S^{n-1})$.
As pointed out in \cite[p.\,224]{cww93}, the level sets associated
with a rough homogeneous kernel are described by starlike sets rather than
by standard Euclidean balls. Thus, a small perturbation may move a point out of the related
level sets involving $\mathcal{M}_\theta^\Omega$ and hence
the ``global'' geometric covering lemma used in
\cite{dlyyz26} cannot be applied directly in this setting.

To overcome this geometric difficulty, we
introduce a class of "bad cubes",
related to a given function, based on the scale of dyadic cubes and
borrow some ideas from the sparse domination to
establish a dyadic decomposition of any nonnegative $L^1$-function. Combining this
dyadic decomposition of the function with a height decomposition of the rough kernel,
we discretize the rough operators into some dyadic functionals and establish some subtle
estimates for their level sets, which further implies
the desired weak-type estimates.

Theorem \ref{thm-diag} yields three interesting applications,
which we describe here.
The way one can handle the rough maximal operator $M^\Omega$
is to dominate it by the operator $M^\ast_\Omega$. Specifically,
given $\Omega\in L^1(\mathbb S^{n-1})$ and
$f \in L_{\rm loc}^1(\mathbb{R}^n)$, by polar coordinates, one has
\begin{align}\label{eq-Ms}
M^\Omega(f)\le M^\ast_\Omega(f):=\int_{\mathbb{S}^{n-1}} |\Omega(\omega)|
M_{\omega}(f)\,d\sigma(\omega),
\end{align}
where, for any $\omega\in \mathbb S^{n-1}$ and $x\in\mathbb{R}^{n}$,
\begin{align*}
M_{\omega}(f)(x) := \sup_{r\in(0,\infty)} \frac{1}{2r} \int_{-r}^r |f(x-s\omega)| \,ds.
\end{align*}
This is the so-called "method of rotations" of Calder\'on and Zygmund in \cite{cz56}.
The strong $(p,p)$ bound of $M^\Omega$ follows easily from
the one of $M_\omega$ for all $p\in(1,\infty]$. However, this
method of proof breaks down for $p=1$ because
the space weak $L^1$ is not normed.
Moreover, R. Fefferman \cite[p.\,176]{f78}
proved that there exists no weak type $(1,1)$ estimate for $M^\ast_\Omega$ in general
even for $\Omega\equiv 1$.

Somewhat surprising, Sj\"ogren and Soria \cite{ss96,ss97} showed that
the operator $M^\ast_\Omega$ is of weak type $(1,1)$ when
the operator is applied to the radial functions only, where $\Omega\in L^1(\mathbb S^{n-1})$.
In \cite{ss96}, a key tool to prove this result in $2$-dimensional case is to determine the following
weak-type estimate:
\begin{align}\label{eq-wH}
\sup_{\lambda\in(0,\infty)}\lambda\int_0^\infty \left|\left\{x\in\mathbb R:
|H_\alpha f(x,y)|>\lambda \right\}\right|y^{\alpha-1}\,dy\le
\|h\|_{L^1(\mathbb R)}\|f\|_{L^1(\mathbb R)},
\end{align}
where $\alpha\in\mathbb R$, $h$ is even, nonnegative, and nonincreasing
on $(0,\infty)$, and the \emph{generalized Poisson integral $H_\alpha$}
is defined by setting, for any $f\in L^1(\mathbb R)$ and $(x,y)\in\mathbb R^{2}_+$,
\begin{align}\label{eq-io}
H_\alpha f(x,y):= \frac{1}{y^{\alpha}}\{[yh(\cdot y)]
\ast f(x)\}.
\end{align}
The estimate \eqref{eq-wH}, when $\alpha=2$, coincides with
\cite[Theorem 1]{ss96} and,
as pointed out in \cite[p.\,228]{ss96}, the estimate \eqref{eq-wH}
for $\alpha\in(1,\infty)$ follows from the same techniques of the proof of \cite[Theorem 1]{ss96}.
For $\alpha\in(-\infty,0)$, the estimate \eqref{eq-wH} can be
concluded by \cite[Theorem 1]{s83} under the additional assumption
$\int_\mathbb{R}h(x)|\log|x||\,dx<\infty$. This has been further  weakened in
\cite{ms95} under the condition $\int_\mathbb{R}h(x)|\log|\log |x|||\,dx<\infty$.
Moreover, Sj\"ogren \cite[p.\,229 and p.\,248]{s81} illustrated
the estimate \eqref{eq-wH} fails for $\alpha\in[0,1]$. In \cite[p.\,228]{ss96},
Sj\"ogren and Soria asked
\emph{whether \eqref{eq-wH} holds
for $\alpha\in(-\infty,0)$ and general $h$.}

In this article, we give an affirmative answer to this open question
by applying Theorem \ref{thm-diag} (see Remark \ref{rem:open_question}).
Indeed, via the change of variables $y=t^{-1}$ and $\gamma:=-\alpha$,
\eqref{eq-wH} is related to \eqref{eq-p=1} in a transparent way [see \eqref{eq-im}].
We show that \eqref{eq-wH} holds for $\alpha\in(-\infty,0)$ without assuming
any logarithmic integrability on $h$ and also extend it
to higher-dimensional cases (see Theorem \ref{thm:rough_application}).

As we mentioned before, one cannot expect weak type
$(1,1)$ estimate for $M^\ast_\Omega$ even for $\Omega\equiv 1$ (see \cite[p.\,178]{f78}).
In this article, we provide an alternative solution without restricting
to radial functions, different from Sj\"ogren and Soria did in \cite{ss96,ss97},
by lifting the operator $M^\ast_\Omega$ into the upper-half plane.
Moreover, we prove in Theorem \ref{thm-stronger-2.5} that, given $\Omega\in L(\log L)(\mathbb S^{n-1})$ and
$\gamma\in (-\infty,-n)\cup (0,\infty)$, there exists a positive constant
$C$, depending only on $n$ and $\gamma$, such that, for any $f\in L^1(\mathbb R^n)$,
\begin{equation}\label{eq-strong-up-p=1}
\sup_{\lambda\in(0,\infty)}
\lambda\underset{\mathcal M^\ast_{\Omega}(f)(x,t)>\lambda t^\gamma}
{\int_{\mathbb R^n}\int_0^\infty}
t^{\gamma-1}\,dt\,dx \le C(1+|\Omega|_{L(\log L)(\mathbb{S}^{n-1})})\|f\|_{L^1(\mathbb{R}^n)},
\end{equation}
where we refer to \eqref{eq-lMw} for the precise definition of
the lifted operator $\mathcal{M}^\ast_\Omega$.

Besides the above two applications, we also obtain a weak-type characterization
of truncated rough singular integrals, which further yields a new characterization of
the Hardy space $H^1(\mathbb{R}^n)$ in terms of truncated rough singular integrals;
see Subsection \ref{s3.3} for details.
Recall that the real-variable theory of the Hardy space
$H^p(\mathbb R^n)$, with $p\in(0,1]$, traces back to Stein and Weiss \cite{sw60},
which was later comprehensively expanded by Fefferman and Stein \cite{fs72}.
Since $H^p(\mathbb R^n)$
serves as a crucial alternative to the Lebesgue space when $p\in(0,1]$
in the context of studying the boundedness of operators,
its real-variable theory has experienced remarkable growth,
demonstrating profound significance in both harmonic analysis
and partial differential equations; see, for example, \cite{c74,clms89,clms93,em94,hs25,mu94,s83}
along with the monographs \cite{lu95,s93}.
For further generalizations regarding diverse variants of
Hardy spaces, we refer to Bui \cite{b14,b26}, Ky \cite{ky14},
Yang et al. \cite{ylk}, Ho \cite{HoKP12,ho15}, Nakai and Sawano \cite{ns12,ns14},
Cruz-Uribe and Wang \cite{cw14}, as well as Sawano et al. \cite{inns23,s17}.

The remainder of this article is organized as follows.

In Section \ref{s2}, we prove Theorem \ref{thm-diag}.
Subsection \ref{s2.1} is devoted to establishing two kinds of dyadic decompositions,
according to the size of $\gamma$, of
given functions based on a class of bad cubes. The target of
Subsection \ref{s2.2} is to prove
a key estimate of some level sets with rough kernels
in terms of given family of dyadic cubes. In Subsection \ref{s2.3},
we show the sufficiency part of Theorem \ref{thm-diag} and, in
Subsection \ref{s2.4}, we aim to prove the necessity part of Theorem \ref{thm-diag}.
We then give some applications of Theorem \ref{thm-diag} in Section \ref{s3}.
In Subsection \ref{s3.1}, using Stein--N. Weiss adding-up lemma,
we obtain the weak-type estimates of a class of integral operators as in
\eqref{eq-io}. Subsection \ref{s3.2} is devoted to establishing
the weak-type estimates \eqref{eq-strong-up-p=1}.
In Subsection \ref{s3.3}, by establishing a new Coltar-type inequality,
we provide a weak-type characterization of
truncated rough singular integrals, which further gives a
characterization of Hardy spaces.

We end this section by introducing several notational conventions.
Let $\mathbb{N}:=\{1,2,\ldots\}$ and
$\mathbb{Z}_+:=\mathbb{N}\cup\{0\}$. The \emph{symbol $C$} always denotes a
positive constant independent of the main parameters involved,
but may vary from line to line. For a set $E \subset \mathbb{R}^n$,
we write $\mathbf{1}_E$ for its \emph{characteristic function}.
Let $V\subset X$ be two sets and we define $V^\complement:=
X\setminus V$. For any $s\in\mathbb R$,
we use $\lceil s \rceil$ to denote
the largest integer not greater than $s$.
The limit $t\to 0^+$ means that there exists
a positive constant $c_0\in(0,\infty)$ such that
$t\in(0,c_0)$ and $t\to 0$.
The \emph{notation $f \lesssim g$} indicates that $f \leq C g$ for some positive constant $C$.
When both $f \lesssim g$ and $g \lesssim f$ hold,
we write $f \sim g$. If $f \leq C g$ and simultaneously $g = h$ or $g \leq h$,
we express this as $f \lesssim g = h$ or $f \lesssim g \leq h$, respectively.
The \emph{cardinality} of a finite set $A$ is denoted by $\sharp A$.
For any bounded set $A\subset\mathbb{R}^n$
with $|A|\in(0,\infty)$ and $f\in L_{\mathrm{loc}}^1$,
let
$$\fint_{A}f(x)\,dx:=\frac{1}{|A|}\int_{A}f(x)\,dx.$$
Finally, in all proofs
we consistently retain the notation introduced in the original theorem (or related statement).

\section{Proof of Theorem \ref{thm-diag}}\label{s2}

In this section, we aim to show Theorem \ref{thm-diag}.
To this end, we first establish the dyadic decompositions of
given functions based on a class of bad cubes in Subsection \ref{s2.1}.
Then, in Subsection \ref{s2.2}, we prove a key estimate of some level sets
with rough kernels in term of
given family of dyadic cubes.
Subsection \ref{s2.3} is devoted to showing the sufficiency part
of Theorem \ref{thm-diag} and we give the proof of
the necessity part of
Theorem \ref{thm-diag} in Subsection \ref{s2.4}.

\subsection{Dyadic Decompositions Based on Bad Cubes}\label{s2.1}
In this subsection, we first introduce a class of bad cubes
related to a given function and then
we establish the dyadic decompositions of functions based on them.

Given
\(D\in\mathbb{R}\) and a nonnegative \(f\in L^1(\mathbb R^n)\),
a dyadic cube \(Q\subset\mathbb R^n\)
is said to be \emph{$(D,f)$-bad} (for short, \emph{bad}) if
\begin{align}\label{2}
\int_Q f(y)\,dy>\ell(Q)^D,
\end{align}
where $\ell(Q)$ denotes the edge length of $Q$.
Let the \emph{symbol \(\mathfrak B:=\mathfrak B(D,f)\)} denote the collection of all
$(D,f)$-bad dyadic cubes.
Note that, if \(Q\in \mathfrak B\), then, by \eqref{2}, we obtain
\[
\ell(Q)^D
<
\int_Q f(y)\,dy
\leq
\|f\|_{L^1(\mathbb R^n)},
\]
which further implies that, when $D>0$, $Q$ is contained in a maximal bad dyadic cube.
Clearly, the maximal bad dyadic
cubes are pairwise disjoint and hence countable.
Moreover, they have the following property.
\begin{lemma}\label{lem-cz-decomposition-positive}
Let $D\in (n,\infty)$ and
$f\in L^1(\mathbb{R}^n)$ be nonnegative.
Let $\mathcal{Q}$ be the collection of all maximal $(D,f)$-bad dyadic cubes.
Then, for almost every $x\in\mathbb R^n$,
\begin{align}\label{4}
f(x)=\sum_{Q\in\mathcal Q}f(x){\mathbf 1}_Q(x).
\end{align}
\end{lemma}

\begin{proof}
Let \(x\) be a Lebesgue point of \(f\) with \(f(x)>0\). Let \(Q_k(x)\)
be the dyadic cube containing \(x\) with edge length \(2^{-k}\). Using the
Lebesgue differentiation theorem, we find that, for \(k\in\mathbb N\) sufficiently large,
\[
\fint_{Q_k(x)}f(y)\,dy>\frac12 f(x)>0.
\]
On the other hand, since \(D-n>0\), it follows that
$\ell(Q_k(x))^{D-n}=2^{-k(D-n)}\to0$
as $k\to\infty.$
Therefore, when \(k\in\mathbb N\) is sufficiently large,
\[
\fint_{Q_k(x)}f(y)\,dy
>
\ell(Q_k(x))^{D-n}.
\]
Thus, \(Q_k(x)\) is bad for all sufficiently large \(k\) and hence
\(x\) is contained in a maximal bad dyadic cube. This proves
\begin{align}\label{3}
\left\{x\in\mathbb{R}^n:f(x)>0\ \text{and}\ x\ \text{is a Lebesgue point of}\ f\right\}
\subset \bigcup_{Q\in\mathcal Q}Q.
\end{align}
Moreover, \eqref{4} follows directly from \eqref{3}
and the disjointness of maximal bad dyadic cubes.
This finishes the proof of Lemma
\ref{lem-cz-decomposition-positive}.
\end{proof}

For negative $D$, we cannot use the minimal bad dyadic cubes to
decompose $f$. Instead, we have the following decomposition,
borrowing some ideas from sparse dominations.

\begin{lemma}\label{lem-cz-decomposition-negative}
Let $D\in (-\infty,0)$ and
$f\in L^1(\mathbb{R}^n)$ be nonnegative.
For any \(Q\in\mathfrak B\), define
\[
E_Q
:=
Q\setminus
\bigcup_{\genfrac{}{}{0pt}{}{R\in\mathfrak B}
{R\subsetneqq Q}}R.
\]
Then
\begin{align}\label{EQ}
E_Q
=
\bigcup_{\genfrac{}{}{0pt}{}{Q'\in\operatorname{ch}(Q)}{Q'\notin\mathfrak B}} Q',
\end{align}
where \(\operatorname{ch}(Q)\) denotes the family of dyadic children of \(Q\). Moreover, the
sets \(\{E_Q\}_{Q\in\mathfrak B}\) are pairwise disjoint and, for almost every
$x\in \mathbb R^n$,
\begin{align}\label{fd}
f(x)
=
\sum_{Q\in\mathfrak B}f(x){\mathbf 1}_{E_Q}(x).
\end{align}
\end{lemma}

\begin{proof}
We claim that the following assertions hold.
\begin{enumerate}[{\rm(i)}]
\item If \(Q\in\mathfrak B\), then its dyadic parent also belongs to
\(\mathfrak B\).

\item For almost every \(x\in\mathbb R^n\) with \(f(x)>0\), there exists
a minimal dyadic cube
\(Q_x\in\mathfrak B\) containing \(x\) such that no strictly
smaller bad dyadic cube containing \(x\) is contained in \(Q_x\).
\end{enumerate}
We first prove (i). Let \(Q\in\mathfrak B\)
and let \(\widehat Q\) be its dyadic parent. Then
\[
\int_{\widehat Q}f(y)\,dy\geq \int_Qf(y)\,dy>\ell(Q)^D.
\]
Since \(D<0\) and \(\ell(\widehat Q)=2\ell(Q)\), it follows that
$\ell(\widehat Q)^D=2^D\ell(Q)^D<\ell(Q)^D.$
Consequently,
\[
\int_{\widehat Q}f(y)\,dy>\ell(\widehat Q)^D
\]
and hence \(\widehat Q\in\mathfrak B\). This shows (i).

We now prove (ii). Let \(x\) be a Lebesgue point of \(f\) with \(f(x)>0\)
and, for any
$k\in\mathbb{Z}$, let \(Q_k(x)\) be the dyadic cube containing \(x\) with edge length
\(2^{-k}\). By the Lebesgue differentiation theorem, we can choose
\(k_0\in\mathbb Z\) such that
$\|f\|_{L^1(Q_{k_0}(x))}>0.$
Since \(D<0\),
it follows that, for all sufficiently small \(k\le k_0\),
\[
\int_{Q_{k}(x)}f(y)\,dy\geq \int_{Q_{k_0}(x)}f(y)\,dy>\ell(Q_k(x))^D;
\]
that is, \(Q_k(x)\in\mathfrak B\) for all sufficiently small \(k\le k_0\).
On the other hand,
by \eqref{2} and $D<0$,
we find that, when \(\ell(Q)\) is sufficiently small,
\[
\int_Qf(y)\,dy\leq \|f\|_{L^1(\mathbb R^n)}<\ell(Q)^D,
\]
which shows that any $Q\in\mathfrak{B}$ contains a minimal
bad dyadic cube and hence the family \(\{Q_k(x)\}_{k\in\mathbb{Z}}
\cap \mathfrak{B}\)
has a minimal
bad descendant. This shows (ii).

We next prove \eqref{EQ}. Let \(Q'\in\operatorname{ch}(Q)\) with
\(Q'\notin\mathfrak B\). If \(Q'\) contains some bad dyadic cube,
then the property (i) proved above implies that
\(Q'\in\mathfrak B\), which contradicts \(Q'\notin\mathfrak B\).
Thus, no bad dyadic subcube of
\(Q\) is contained in \(Q'\) and hence
$Q'\subset E_Q.$
This shows
\[
\bigcup_{\genfrac{}{}{0pt}{}{Q'\in\operatorname{ch}(Q)}{Q'\notin\mathfrak B}} Q'
\subset E_Q.
\]
Conversely, let \(x\in E_Q\). Let \(Q_x'\) be the unique dyadic child of
\(Q\) containing \(x\). If \(Q_x'\in\mathfrak B\), then \(x\) belongs to a
strict bad subcube of \(Q\), which contradicts \(x\in E_Q\).
Hence, \(Q_x'\notin\mathfrak B\) and
\[
x\in Q_x'\subset
\bigcup_{\genfrac{}{}{0pt}{}{Q'\in\operatorname{ch}(Q)}{Q'\notin\mathfrak B}} Q'.
\]
This proves the reverse inclusion and hence \eqref{EQ}.

We now show that the sets \(\{E_Q\}_{Q\in\mathfrak B}\) are pairwise disjoint.
Indeed, if \(Q,R\in\mathfrak B\) with $Q\neq R$, then either they
are disjoint, or one of them is strictly contained in the other. If, for
instance, \(R\subsetneqq Q\), then, by the definition of \(E_Q\), we find that
$E_Q\cap R=\emptyset.$
Since \(E_R\subset R\), we infer that \(E_Q\cap E_R=\emptyset\).
Thus, the sets \(\{E_Q\}_{Q\in\mathfrak B}\) are pairwise
disjoint.

Finally, let \(x\) be a Lebesgue point of \(f\) with \(f(x)>0\). By (ii),
we conclude that
there exists a minimal bad dyadic cube \(Q_x\in\mathfrak B\) containing \(x\),
which further implies that, for any $R\in\mathfrak{B}$ and $R\subsetneqq Q_x$,
$x\notin R$. Consequently, \(x\in E_{Q_x}\). This implies that
\[
\left\{x\in\mathbb R^n:f(x)>0\ \text{and}\ x\ \text{is a Lebesgue point of}\ f\right\}
\subset
\bigcup_{Q\in\mathfrak B}E_Q.
\]
Combining this with the disjointness of the sets \(\{E_Q\}_{Q\in\mathfrak{B}}\), we obtain
\eqref{fd}. This finishes the proof of Lemma
\ref{lem-cz-decomposition-negative}.
\end{proof}

\subsection{A Key Estimate for Dyadic Functionals Involving Rough Kernels}\label{s2.2}

In this subsection, our target is to establish
a key estimate of some level sets in terms of given family of dyadic cubes.
Precisely speaking,
the following lemma is the key ingredient of the proof of Theorem \ref{thm-diag}.
For any $\gamma\in\mathbb{R}$, we write
\begin{align}\label{dv}
d\nu_\gamma(x,t):=t^{\gamma-1}\,dt\,dx.
\end{align}
\begin{lemma}\label{lem-main}
Let \(D:=n+\gamma\) with \(\gamma\in\mathbb{R}\),
$\Omega\in L(\log L)(\mathbb S^{n-1})$ be
nonnegative, and
\(\mathcal Q\) be a family of dyadic cubes.
For each \(Q\in\mathcal Q\), let \(\mu_Q\) be a measure supported
in \(Q\) such that
\begin{align}\label{9}
\mu_Q(Q)\leq \ell(Q)^D.
\end{align}
Then the following statements hold.
\begin{enumerate}[{\rm (i)}]
\item If $\gamma\in(-\infty,-n)$, then there exists a positive constant $C$, depending only on $n$,
such that, for any \(\beta\in[1,\infty)\),
\begin{align}\label{e5.1}
\nu_\gamma\left(\left\{(x,t)\in\mathbb{R}^{n+1}_+:\ H_{\rm large}(x,t)>C\beta\right\}\right)
\lesssim
\frac{|\Omega|_{L(\log L)(\mathbb S^{n-1})}}{\beta}
\sum_{Q\in\mathcal Q}\ell(Q)^D,
\end{align}
where
\[
H_{\rm large}(x,t)
:=
\sum_{\{Q\in\mathcal Q:\,t\le \ell(Q)\}}
t^{-D}
\int_{Q\cap B(x,2t)}
\Omega\left(\frac{x-y}{|x-y|}\right)
\,d\mu_Q(y)
\]
and the implicit positive constant depends only on $n$.

\item If the dyadic cubes in $\mathcal Q$ are pairwise disjoint and $\gamma\in(0,\infty)$,
then there exists a positive constant $C$, depending only on $n$,
such that, for any \(\beta\in[1,\infty)\),
\begin{align}\label{7}
\nu_\gamma\left(\left\{(x,t)\in\mathbb{R}^{n+1}_+:\ H_{\rm small}(x,t)>C\beta\right\}\right)
\lesssim
\frac{|\Omega|_{L(\log L)(\mathbb S^{n-1})}}{\beta}
\sum_{Q\in\mathcal Q}\ell(Q)^D,
\end{align}
where \[
H_{\rm small}(x,t)
:=
t^{-D}\sum_{\{Q\in\mathcal Q:\,\ell(Q)\le t\}}
\int_{Q\cap B(x,2t)}
\Omega\left(\frac{x-y}{|x-y|}\right)
\,d\mu_Q(y)
\]
and the implicit positive constant depends only on $n$.
\end{enumerate}
\end{lemma}

\begin{proof}
We first prove (i). Fix $(x,t)\in\mathbb{R}^{n+1}_+$.
Note that
\begin{align}\label{e5.2}
\sum_{\genfrac{}{}{0pt}{}{\{Q\in\mathcal Q:\,t\le\ell(Q)}{Q\cap B(x,2t)\neq\emptyset\}}}
\ell(Q)^D&=\sum_{j=0}^\infty \sum_{\genfrac{}{}{0pt}{}{\{Q\in\mathcal Q:\, 2^j t\leq \ell(Q)<2^{j+1}t}
{Q\cap B(x,2t)\neq\emptyset\}}}
\ell(Q)^D\nonumber\\
&\le\sum_{j=0}^\infty 2^{jD}\sum_{\genfrac{}{}{0pt}{}{\{Q\in\mathcal Q:\,2^j t\leq \ell(Q)<2^{j+1}t}
{Q\cap B(x,2t)\neq\emptyset\}}}
t^D\sim t^D,
\end{align}
where the last step used \(D<0\) and $\sharp\{Q\in\mathcal Q:\, 2^j t\leq \ell(Q)<2^{j+1}t,
Q\cap B(x,2t)\neq\emptyset\}\sim 1$.

Now, fix \(\beta\geq1\). Decompose the function $\Omega$ as
\begin{align}\label{lower-h}
\Omega_{-1}:=\Omega{\mathbf 1}_{\{\Omega\leq\beta\}},
\end{align}
and, for any \(k\in\mathbb Z_+\),
\begin{align}\label{high-h}
\Omega_k
:=
\Omega{\mathbf 1}_{\{2^k\beta<\Omega\leq 2^{k+1}\beta\}}.
\end{align}

For the lower-height part of $\Omega$,
let \(H_{\Omega_{-1}}\) be the part of \(H_{\rm large}\) with \(\Omega\) replaced by \(\Omega_{-1}\).
Since \(\Omega_{-1}\leq \beta\),
from \eqref{9} and \eqref{e5.2}, it follows that
\begin{align}\label{e5.3}
H_{\Omega_{-1}}(x,t)
&\leq
\beta t^{-D}
\sum_{\genfrac{}{}{0pt}{}{\{Q\in\mathcal Q:\, t\leq \ell(Q)}{Q\cap B(x,2t)\neq\emptyset\}}}
\mu_Q(Q)
\nonumber\\
&\leq
\beta t^{-D}
\sum_{\genfrac{}{}{0pt}{}{\{Q\in\mathcal Q:\, t\leq \ell(Q)}{Q\cap B(x,2t)\neq\emptyset\}}}
\ell(Q)^D\lesssim
\beta t^{-D}t^D=
\beta.
\end{align}

We now consider the high-height part of $\Omega$. For any \(k,j\in\mathbb Z_+\), define
\[
H_{k,j}(x,t)
:=
\sum_{\genfrac{}{}{0pt}{}{\{Q\in\mathcal Q:}{2^j t\leq \ell(Q)<2^{j+1}t\}}}
t^{-D}
\int_{Q\cap B(x,2t)}
\Omega_k\left(\frac{x-y}{|x-y|}\right)
\,d\mu_Q(y).
\]
By \eqref{9} and \eqref{e5.2} again, we find that, for any \(k,j\in\mathbb Z_+\),
\begin{align}\label{e5.4}
H_{k,j}(x,t)
&\leq
2^{k+1}\beta\,t^{-D}
\sum_{\genfrac{}{}{0pt}{}{\{Q\in\mathcal Q:\, 2^jt\leq \ell(Q)<2^{j+1}t}{Q\cap B(x,2t)\neq\emptyset\}}}
\mu_Q(Q)
\nonumber\\
&\leq
2^{k+1}\beta\,t^{-D}
\sum_{\genfrac{}{}{0pt}{}{\{Q\in\mathcal Q:\, 2^jt\leq \ell(Q)<2^{j+1}t}{Q\cap B(x,2t)\neq\emptyset\}}}
\ell(Q)^D
\nonumber\\
&\lesssim
2^k\beta\,t^{-D}(2^jt)^D=
2^k\beta\,2^{jD}.
\end{align}
For any $k\in\mathbb Z_+$, choose
$J_k:=\left\lceil \frac{2(k+1)}{-D}\right\rceil .$
Then \eqref{e5.4} implies
\begin{align}\label{e5.5}
\sum_{k=0}^\infty\sum_{j>J_k}H_{k,j}(x,t)
\lesssim
\beta
\sum_{k=0}^\infty
2^k 2^{J_kD}
\lesssim
\beta.
\end{align}

Therefore, we only need to consider the finite part
\[
N(x,t)
:=
\sum_{k=0}^\infty
\sum_{0\leq j\leq J_k}
H_{k,j}(x,t).
\]
Combining this with \eqref{e5.3} and \eqref{e5.5}, we have
$H(x,t)\leq C\beta+N(x,t),$
where the positive constant $C$ depends only on $n$.
Hence, by Chebyshev's inequality, we conclude that
\begin{align}\label{eq-cc}
\nu_\gamma\left(\left\{H>2C\beta\right\}\right)
\leq
\nu_\gamma\left(\left\{N>C\beta\right\}\right)\lesssim
\frac1\beta
\iint_{\mathbb{R}^{n+1}_+} N(x,t)\,d\nu_\gamma(x,t).
\end{align}
Consequently, it remains to estimate the $L^1(\mathbb R^{n+1}_+,\nu_\gamma)$-norm of
$N$. To this end, fix \(k,j\in\mathbb{Z}_+\). By Tonelli's theorem, we obtain
\begin{align}\label{13}
&\iint_{\mathbb{R}^{n+1}_+} H_{k,j}(x,t)\,d\nu_\gamma(x,t)\nonumber\\
&\quad\leq
\sum_{Q\in\mathcal Q}
\int_Q
\int_{2^{-j-1}\ell(Q)}^{2^{-j}\ell(Q)}
t^{-D+\gamma-1}
\int_{|x-y|<2t}
\Omega_k\left(\frac{x-y}{|x-y|}\right)
\,dx\,dt\,d\mu_Q(y).
\end{align}
Using polar coordinates, we find that
\[
\int_{|x-y|<2t}
\Omega_k\left(\frac{x-y}{|x-y|}\right)
\,dx
\lesssim
t^n
\|\Omega_k\|_{L^1(\mathbb S^{n-1})},
\]
which, combined with \eqref{9} and \eqref{13}, further implies that
\begin{align}\label{14}
\iint_{\mathbb{R}^{n+1}_+} H_{k,j}(x,t)\,d\nu_\gamma(x,t)
\lesssim
\|\Omega_k\|_{L^1(\mathbb S^{n-1})}
\sum_{Q\in\mathcal Q}\mu_Q(Q)
\lesssim
\|\Omega_k\|_{L^1(\mathbb S^{n-1})}
\sum_{Q\in\mathcal Q}\ell(Q)^D.
\end{align}
Summing \eqref{14} over \(0\leq j\leq J_k\) and
\(k\geq0\) yields that
\begin{align*}
\iint_{\mathbb{R}^{n+1}_+} N(x,t)\,d\nu_\gamma(x,t)
&\lesssim
\sum_{k=0}^\infty
(J_k+1)\|\Omega_k\|_{L^1(\mathbb S^{n-1})}
\sum_{Q\in\mathcal Q}\ell(Q)^D
\\
&\lesssim
\sum_{k=0}^\infty
(k+1)\|\Omega_k\|_{L^1(\mathbb S^{n-1})}
\sum_{Q\in\mathcal Q}\ell(Q)^D.
\end{align*}
Since \(\beta\geq1\) on the support of \(\Omega_k\), we infer that
$ 2^k\leq \Omega$ and hence
$k+1\lesssim 1+\log^+\Omega.$
Thus,
\[
\sum_{k=0}^\infty (k+1)\|\Omega_k\|_{L^1(\mathbb S^{n-1})}
\lesssim
\int_{\mathbb S^{n-1}}
\Omega(\theta)
\bigl(1+\log^+\Omega(\theta)\bigr)
\,d\sigma(\theta)
=
|\Omega|_{L(\log L)(\mathbb{S}^{n-1})}.
\]
Consequently,
\begin{align}\label{15}
\iint_{\mathbb{R}^{n+1}_+} N(x,t)\,d\nu_\gamma(x,t)
\lesssim
|\Omega|_{L(\log L)(\mathbb{S}^{n-1})}
\sum_{Q\in\mathcal Q}\ell(Q)^D.
\end{align}
Substituting this into \eqref{eq-cc}, we obtain \eqref{e5.1} and hence
finish the proof of (i).

The proof of (ii) is similar to the one of (i) and hence
we only highlight the difference from (i).

First, we note that, given \((x,t)\in \mathbb{R}^{n+1}_+\) and $\alpha\in(0,1]$,
\begin{align}\label{8}
\sum_{\genfrac{}{}{0pt}{}{\{Q\in\mathcal Q:\, \ell(Q)\leq \alpha t}{Q\cap B(x,2t)\neq\emptyset\}}}
\ell(Q)^D
\lesssim
\alpha^\gamma t^D,
\end{align}
where the implicit positive constant depends only on $n$.
Indeed, given $Q\in\mathcal{Q}$,
if \(Q\cap B(x,2t)\neq\emptyset\) and \(\ell(Q)\leq \alpha t\), then
$Q\subset B(x,3t).$
Since the cubes in \(\mathcal Q\) are pairwise disjoint, it follows that
\[
\sum_{\genfrac{}{}{0pt}{}{\{Q\in\mathcal{Q}:\, \ell(Q)\leq\alpha t}{Q\cap B(x,2t)\neq\emptyset\}}}
|Q|
\leq
|B(x,3t)|
\lesssim t^n.
\]
Therefore,
\[
\begin{aligned}
\sum_{\genfrac{}{}{0pt}{}{\{Q\in\mathcal{Q}:\,\ell(Q)\leq\alpha t}{Q\cap B(x,2t)\neq\emptyset\}}}
\ell(Q)^D
&=
\sum_{\genfrac{}{}{0pt}{}{\{Q\in\mathcal{Q}:\, \ell(Q)\leq\alpha t}{Q\cap B(x,2t)\neq\emptyset\}}}
\ell(Q)^\gamma |Q|
\leq
(\alpha t)^\gamma
\sum_{\genfrac{}{}{0pt}{}{\{Q\in\mathcal{Q}:\, \ell(Q)\leq\alpha t}{Q\cap B(x,2t)\neq\emptyset\}}}
|Q|\lesssim \alpha^\gamma
t^\gamma t^n
=
\alpha^\gamma t^D.
\end{aligned}
\]
This proves \eqref{8}.

Fix \(\beta\geq1\). Decompose the kernel $\Omega$ as in
\eqref{lower-h} and \eqref{high-h}.
For the lower-height part of $\Omega$,
let \(H_{\Omega_{-1}}\) be the part of \(H_{\rm small}\) with \(\Omega\) replaced by \(\Omega_{-1}\).
Since \(\Omega_{-1}\leq\beta\), from \eqref{9} and \eqref{8},
we deduce that, for any $(x,t)\in \mathbb{R}^{n+1}_+$,
\begin{align}\label{10}
H_{\Omega_{-1}}(x,t)
\leq
\beta t^{-D}
\sum_{\genfrac{}{}{0pt}{}{\{Q\in\mathcal{Q}:\, \ell(Q)\leq t}{Q\cap B(x,2t)\neq\emptyset\}}}
\mu_Q(Q)
\leq
\beta t^{-D}
\sum_{\genfrac{}{}{0pt}{}{\{Q\in\mathcal{Q}:\, \ell(Q)\leq t}{Q\cap B(x,2t)\neq\emptyset\}}}
\ell(Q)^D
\lesssim
\beta t^{-D}t^D
\lesssim
\beta.
\end{align}
For the higher-hight part of $\Omega$, given \(k,j\in\mathbb{Z}_+\) and
$(x,t)\in \mathbb{R}^{n+1}_+$, define
\[
H_{k,j}(x,t)
:=t^{-D}
\sum_{\genfrac{}{}{0pt}{}{\{Q\in\mathcal Q:}{2^{-j-1}t<\ell(Q)\leq2^{-j}t\}}}
\int_{Q\cap B(x,2t)}
\Omega_k\left(\frac{x-y}{|x-y|}\right)
\,d\mu_Q(y).
\]
Since $\Omega_k\leq2^{k+1}\beta,$ from \eqref{9} and \eqref{8},
it follows that, for any $(x,t)\in \mathbb{R}^{n+1}_+$,
\begin{align}\label{11}
H_{k,j}(x,t)
&\leq
2^{k+1}\beta\, t^{-D}
\sum_{\genfrac{}{}{0pt}{}{\{Q\in\mathcal{Q}:\, 2^{-j-1}t<\ell(Q)\leq2^{-j}t,}{Q\cap B(x,2t)\neq\emptyset\}}}
\mu_Q(Q)\nonumber\\
&\leq
2^{k+1}\beta\, t^{-D}
\sum_{\genfrac{}{}{0pt}{}{\{Q\in\mathcal{Q}:\, 2^{-j-1}t<\ell(Q)\leq2^{-j}t,}{Q\cap B(x,2t)\neq\emptyset\}}}
\ell(Q)^D\nonumber\\
&\lesssim 2^{k+1}\beta\, t^{-D}2^{-j\gamma} t^D\sim2^k\beta\,2^{-j\gamma}.
\end{align}
For any $k\in\mathbb Z_+$, choose
$J_k:=\left\lceil \frac{2(k+1)}{\gamma}\right\rceil.$
Then \eqref{11} yields that
\begin{align}\label{12}
\sum_{k=0}^\infty\sum_{j>J_k}H_{k,j}(x,t)
\lesssim
\beta
\sum_{k=0}^\infty
2^k
\sum_{j>J_k}2^{-j\gamma}
\lesssim
\beta
\sum_{k=0}^\infty
2^k2^{-J_k\gamma}
\lesssim
\beta.
\end{align}
Hence the only remaining part is the sum
\[
N:=
\sum_{k=0}^\infty
\sum_{0\leq j\leq J_k}
H_{k,j}.
\]
Repeating the proof of \eqref{15} and combining
\eqref{10} and \eqref{12}, we then obtain the desired
estimate \eqref{7}.
This finishes the proof of Lemma \ref{lem-main}.
\end{proof}

\subsection{Sufficiency of Theorem \ref{thm-diag}}\label{s2.3}

In this subsection, we aim to show the sufficiency of Theorem \ref{thm-diag}.

Given \(\Omega\in L^1(\mathbb S^{n-1})\) and \(f\in L^1(\mathbb R^n)\), define
$\mathcal M^\Omega(f):=\mathcal M^\Omega_1(f)$;
that is, for any $(x,t)\in
\mathbb R^{n+1}_+$,
\begin{align}\label{eq-M}
\mathcal M^\Omega(f)(x,t)
:=
\sup_{r\in[t,2t)}
\frac1{|B(x,r)|}
\int_{B(x,r)}
|f(y)|
\left|
\Omega\left(\frac{x-y}{|x-y|}\right)
\right|
\,dy.
\end{align}
Given $\theta\in(0,\infty)$ and
$f\in L^1_{\rm loc}(\mathbb R^{n})$, by the definition of $\mathcal{M}_\theta^\Omega$,
we find that, for any $(x,t)\in\mathbb R^{n+1}_+$,
$\mathcal{M}_\theta^\Omega(f)(x,t)=\mathcal{M}^\Omega(f)(x,\theta t)$ and hence
a change of variables yields that, for any $p\in[1,\infty)$ and $\gamma\in \mathbb R$,
\begin{align*}
\left\| \frac{\mathcal{M}^\Omega_\theta(f)(x,t)}{t^\frac{\gamma}{p}}\right\|_{L^{p,\infty}
(\mathbb R^{n+1}_+,\nu_\gamma)}
:=\,&\sup_{\lambda\in(0,\infty)}\lambda
\left[\nu_\gamma\left(
\left\{(x,t)\in \mathbb R^{n+1}_+:
\frac{\mathcal{M}^\Omega_\theta(f)(x,t)}{t^\frac{\gamma}{p}}>\lambda \right\}
\right)\right]^{\frac1p}\\
=\,& \left\|
\frac{\mathcal{M}^\Omega(f)(x,t)}{t^\frac{\gamma}{p}}\right\|_{L^{p,\infty}
(\mathbb R^{n+1}_+,\nu_\gamma)},
\end{align*}
where $\nu_\gamma$ is as in \eqref{dv}.
Thus, we only need to consider $\mathcal{M}^\Omega$.
The proof of the lower estimates are relatively direct, which follows from
the following proposition.
\begin{proposition}\label{thm-lower}
Let \(n\in\mathbb{N}\cap[2,\infty)\), $p\in[1,\infty)$,
\(\gamma\in\mathbb R\setminus\{0\}\), and
\(\Omega\in L^1(\mathbb S^{n-1})\) with \(\Omega\not\equiv0\). Then, for any
\(f\in L^p(\mathbb R^n)\),
\begin{align*}
\liminf_{\lambda\to L_\gamma}
\lambda^p
\underset{\mathcal M^\Omega(f)(x,t)>\lambda t^{\frac{\gamma}{p}}}
{\int_{\mathbb R^n}\int_0^\infty}
t^{\gamma-1}\,dt\,dx
\geq
\frac{c_\Omega^p}{|\gamma|}
\|f\|_{L^p(\mathbb R^n)}^p,
\end{align*}
where $L_\gamma:=0$ for any $\gamma\in(-\infty,0)$ and $L_\gamma:=\infty$
for any $\gamma\in(0,\infty)$ and where
\begin{align*}
c_\Omega
:=
\fint_{B({\bf 0},1)}
\left|
\Omega\left(\frac{z}{|z|}\right)
\right|
\,dz>0
\end{align*}
with ${\bf 0}$ denoting the origin of $\mathbb R^n$.
\end{proposition}

\begin{proof}
Fix $f\in L^p(\mathbb R^n)$. Observe that, for any \((x,t)\in\mathbb R^{n+1}_+\),
\begin{align}\label{eq-At-less-M}
A_t f(x)
:=
\fint_{B(x,t)}
f(y)\,
\left|\Omega\left(\frac{x-y}{|x-y|}\right)\right|
\,dy \le \mathcal M^\Omega(f)(x,t).
\end{align}
Since \(A_t\) is an approximation of the identity, it follows that
\begin{align*}
\lim_{t\to0^+}\left\|A_tf-c_\Omega f\right\|_{L^p(\mathbb{R}^n)}=0,
\end{align*}
where $t\to0^+$ means that $t \in (0,c_0)$
for some positive constant $c_0$ and $t \to 0$.

Applying this and \cite[Lemma 3.3]{dlyyz26}
and using \eqref{eq-At-less-M}, we obtain
\begin{align*}
&\liminf_{\lambda\to L_\gamma}
\lambda^p
\underset{\mathcal M^\Omega(f)(x,t)>\lambda t^{\frac{\gamma}{p}}}
{\int_{\mathbb R^n}\int_0^\infty}
t^{\gamma-1}\,dt\,dx
\\
&\quad\geq
\liminf_{\lambda\to L_\gamma}
\lambda^p
\underset{A_t f(x)>\lambda t^{\frac{\gamma}{p}}}
{\int_{\mathbb R^n}\int_0^\infty}
t^{\gamma-1}\,dt\,dx\geq
\frac{c_\Omega^p}{|\gamma|}
\|f\|_{L^p(\mathbb R^n)}^p.
\end{align*}
This finishes the
proof of Theorem \ref{thm-lower}.
\end{proof}

It remains to prove the upper estimates of Theorem \ref{thm-diag}.
In the case $p\in(1,\infty)$,
the upper estimate follows from Theorem \ref{thm-classical-rough-maximal}:
\begin{align}\label{eq-up-p>1}
&\sup_{\lambda\in(0,\infty)}
\lambda^p
\underset{\mathcal M^\Omega(f)(x,t)>\lambda t^{\frac{\gamma}{p}}}
{\int_{\mathbb R^n}\int_0^\infty}
t^{\gamma-1}\,dt\,dx\nonumber\\
&\quad\le\sup_{\lambda\in(0,\infty)}
\lambda^p
\underset{ M^\Omega(f)(x)>\lambda t^{\frac{\gamma}{p}}}
{\int_{\mathbb R^n}\int_0^\infty}
t^{\gamma-1}\,dt\,dx
\sim \left\|M^\Omega(f)\right\|^p_{L^p(\mathbb{R}^n)}\lesssim
\|f\|^p_{L^p(\mathbb{R}^n)},
\end{align}
where the second step used $\gamma\neq 0$
and the last step used Theorem \ref{thm-classical-rough-maximal}.

The most involved part is the upper estimate
in the case $p=1$, presented in the following theorem.

\begin{theorem}\label{thm-up-p=1}
Let $n\in \mathbb{N}\cap[2,\infty)$, $\gamma\in(-\infty,-n)\cup(0,\infty)$,
and $\Omega \in L(\log L)(\mathbb S^{n-1})$.  Then there exists a positive constant $C$,
depending only on $n$ and $\gamma$, such that,
for any $f\in L^1(\mathbb R^n)$,
\[
\sup_{\lambda\in(0,\infty)}
\lambda
\underset{\mathcal M^\Omega(f)(x,t)>\lambda t^\gamma}
{\int_{\mathbb R^n}\int_0^\infty}
t^{\gamma-1}\,dt\,dx
\leq
C\left[1+|\Omega|_{L(\log L)(\mathbb S^{n-1})}\right]\|f\|_{L^1(\mathbb R^n)}.
\]
\end{theorem}

\begin{proof}
Fix $f\in L^1(\mathbb R^{n})$. Without loss of generality, we may assume that
both $f$ and $\Omega$ are nonnegative. Given $(x,t)\in\mathbb{R}^{n+1}_+$,
since \(r\in[t,2t)\), it follows that
\[
t^{-\gamma}\mathcal M^\Omega(f)(x,t)
\leq
C_n t^{-D}
\int_{|x-y|<2t}
f(y)
\Omega\left(\frac{x-y}{|x-y|}\right)
\,dy=:C_nT f(x,t),
\]
where $D:=n+\gamma$ and the positive constant $C_n$ depends only on $n$.
Therefore, it is enough to prove that, for any $\lambda\in(0,\infty)$,
\begin{align}\label{1}
\nu_\gamma\left(\left\{(x,t)\in\mathbb R^{n+1}_+:T f(x,t)>\lambda\right\}\right)
\lesssim\left[1+|\Omega|_{L(\log L)(\mathbb S^{n-1})}\right]
\frac{\|f\|_{L^1(\mathbb R^n)}}{\lambda},
\end{align}
where the implicit positive constant depends only on $n$ and $\gamma$.

Fix \(\lambda\in(0,\infty)\). Let \(\eta\in(0,\infty)\) be a small constant to be chosen later.
We consider the following two cases for $\gamma$.

\emph{Case 1: $\gamma\in(0,\infty)$}.
Let \(\mathcal Q\) be the collection of all
maximal $(D,\frac{f}{\eta\lambda})$-bad dyadic cubes.
For each \(Q\in\mathcal Q\), define
\[
Q^*
:=
\{x\in\mathbb R^n:\operatorname{dist}(x,Q)<2\ell(Q)\}.
\]
Define
\[
\mathcal B
:=
\bigcup_{Q\in\mathcal Q}
Q^*\times(0,\ell(Q)).
\]
Then we divide the left-hand side of \eqref{1} into the following two parts:
\begin{align}\label{16}
\nu_\gamma\left(\left\{(x,t)\in\mathbb{R}^{n+1}_+:T f(x,t)>\lambda\right\}\right)\le
\nu_\gamma(\mathcal{B}) + \nu_\gamma\left(\left\{(x,t)
\in\mathcal{B}^\complement:T f(x,t)>\lambda\right\}\right).
\end{align}
For the first part, by the disjointness of maximal bad dyadic cubes, \eqref{2}, and $\gamma>0$,
we conclude that
\begin{align}\label{17}
\nu_\gamma(\mathcal B)
&\leq
\sum_{Q\in\mathcal Q}
|Q^*|
\int_0^{\ell(Q)}t^{\gamma-1}\,dt
\lesssim
\sum_{Q\in\mathcal Q}
\ell(Q)^D
<\frac{1}{\eta \lambda}
\sum_{Q\in\mathcal Q}\int_{Q}f(y)\,dy\le \frac{\|f\|_{L^1(\mathbb{R}^n)}}{\eta \lambda}.
\end{align}

It remains to estimate the second part.
For any given $(x,t)\in\mathbb{R}^{n+1}_+$,
using Lemma \ref{lem-cz-decomposition-positive},
we rewrite $Tf(x,t)$ as
\begin{align*}
Tf(x,t)=t^{-D}\sum_{Q\in\mathcal{Q}}
\int_{B(x,2t)\cap Q}
f(y)
\Omega\left(\frac{x-y}{|x-y|}\right)
\,dy.
\end{align*}
Suppose \((x,t)\notin\mathcal B\), \(y\in Q\) with some $Q\in\mathcal{Q}$, and $|x-y|<2t.$
Then
\begin{align}\label{5}
\ell(Q)\leq t.
\end{align}
Indeed, if \(t<\ell(Q)\), then
$\operatorname{dist}(x,Q)\leq |x-y|<2t<2\ell(Q),$
which yields \(x\in Q^*\) and hence
$(x,t)\in Q^*\times(0,\ell(Q))\subset\mathcal B,$
which contradicts $(x,t)\notin \mathcal{B}$.
Thus, \eqref{5} holds. By \eqref{5}, we find that, for any $(x,t)\in \mathcal B^\complement$,
\begin{align}\label{6}
\frac1{2^D\eta\lambda}T f(x,t)
\leq t^{-D}\sum_{Q\in\mathcal Q,\ell(Q)\le t}
\int_{B(x,2t)}
\Omega\left(\frac{x-y}{|x-y|}\right)
\,d\mu_Q(y)=:
H(x,t),
\end{align}
where the measure \(\mu_Q\) is defined by setting
\[
d\mu_Q(y)
:=
\frac1{2^D\eta\lambda}f(y){\mathbf 1}_Q(y)\,dy.
\]
Note that, for any $Q\in\mathcal{Q}$,
$\mu_Q$ satisfies \eqref{9}.
Indeed, since $Q\in\mathcal{Q}$ is a maximal bad dyadic cube,
it follows that the parent of $Q$, denoted by $\widehat{Q}$, does not satisfy \eqref{2}; that is,
\[
\int_Q f(y)\,dy\le \int_{\widehat{Q}}f(y)\,dy
\leq\eta\lambda \ell(\widehat{Q})^D=
2^D\eta\lambda \ell(Q)^D.
\]
Thus, $\mu_Q$ satisfies \eqref{9}.

Now, we apply Lemma \ref{lem-main}(ii) to complete
the proof of Theorem \ref{thm-up-p=1} in this case.
By \eqref{6}, for any $(x,t)\in\mathcal{B}^\complement$, if \(T f(x,t)>\lambda\), then
$H(x,t)>\frac1{2^D\eta}.$
Let \(C\) be the same positive constant as in \eqref{7}.
Choose \(\eta\in(0,\infty)\) sufficiently small such that
$\beta:=\frac1{2^D C\eta}\geq1.$
Then
$\mathcal B^\complement\cap\{T f>\lambda\}\subset\{H>C\beta\}.$
From this, Lemma \ref{lem-main}(ii), and \eqref{2}, we infer that
\begin{align*}
\nu_\gamma\left(\mathcal B^\complement\cap\{T f>\lambda\}\right)
&\lesssim
\frac{|\Omega|_{L(\log L)(\mathbb{S}^{n-1})}}{\beta}
\sum_{Q\in\mathcal Q}\ell(Q)^D\sim
\eta|\Omega|_{L(\log L)(\mathbb{S}^{n-1})}
\sum_{Q\in\mathcal Q}\ell(Q)^D\\
&\le\eta|\Omega|_{L(\log L)(\mathbb{S}^{n-1})}
\frac{1}{\eta \lambda}
\sum_{Q\in\mathcal Q}\int_{Q}f(y)\,dy\le
|\Omega|_{L(\log L)(\mathbb{S}^{n-1})}\frac{\|f\|_{L^1(\mathbb{R}^n)}}{\lambda}.
\end{align*}
Combining this, \eqref{16}, and \eqref{17}, we obtain
\[
\nu_\gamma\{T f>\lambda\}
\lesssim\left[1+|\Omega|_{L(\log L)(\mathbb{S}^{n-1})}\right]
\frac{\|f\|_{L^1(\mathbb{R}^n)}}{\lambda},
\]
where the implicit positive constant
depends only on $n$ and $\gamma$.
This proves \eqref{1} in this case.

\emph{Case 2: $\gamma\in(-\infty,-n)$}.
Let \(\mathfrak B\) denote the collection of all $(D,\frac{f}{\eta\lambda})$-bad dyadic cubes.
For each \(Q\in\mathfrak B\), define
\[
\mathcal U_Q
:=
\left\{(x,t)\in\mathbb R^{n+1}_+:
t>\ell(Q),\ \operatorname{dist}(x,Q)<2t\right\}.
\]
Let
\[
\mathcal U:=\bigcup_{Q\in\mathfrak B}\mathcal U_Q.
\]
Note that, for any \(t>\ell(Q)\),
$|\{x\in\mathbb{R}^n:\operatorname{dist}(x,Q)<2t\}|\lesssim t^n$,
where the implicit positive constant depends only on $n$.
Let \(\mathfrak B_{\min}\) denote the collection of all
minimal bad cubes
with respect to the set inclusion.
Then we have
\begin{align}\label{e3.1}
\nu_\gamma(\mathcal U)
&\leq
\sum_{Q\in\mathfrak B}\nu_\gamma(\mathcal U_Q)
\lesssim
\sum_{Q\in\mathfrak B}\ell(Q)^D\nonumber\\
&\le \sum_{R\in\mathfrak B_{\min}}\sum_{Q\supset R}\ell(Q)^D
= \sum_{R\in\mathfrak B_{\min}}\ell(R)^D\sum_{j=0}^\infty2^{jD}\lesssim
\frac{\|f\|_{L^1(\mathbb R^n)}}{\eta\lambda},
\end{align}
where the second and the last steps used $D<0$.

Now, suppose that \((x,t)\notin\mathcal U\), \(y\in E_Q\subset Q\), and
$|x-y|<2t.$
Then necessarily
\begin{align}\label{e3.2}
t\leq \ell(Q).
\end{align}
Indeed, if \(t>\ell(Q)\), then
$\operatorname{dist}(x,Q)\leq |x-y|<2t,$
which further implies that \((x,t)\in\mathcal U_Q\subset\mathcal U\)
and hence contradicts $(x,t)\notin\mathcal{U}$.
Thus, using Lemma \ref{lem-cz-decomposition-negative}, \eqref{e3.2},
and the definition of \(T\), we conclude that,
for any \((x,t)\in\mathcal U^\complement\),
\begin{align}\label{e4.2}
\frac1{2^{-\gamma}\eta\lambda}T f(x,t)
\leq
H(x,t),
\end{align}
where
\[
H(x,t)
:=
\sum_{Q\in\mathfrak B,\,t\le \ell(Q)}
t^{-D}
\int_{Q\cap B(x,2t)}
\Omega\left(\frac{x-y}{|x-y|}\right)
\,d\mu_Q(y)
\]
and $d\mu_Q(y):=\frac1{2^{-\gamma}\eta\lambda} f(y){\bf 1}_{E_Q}(y)\,dy.$
Note that, from \eqref{EQ}, it follows that, for any $Q\in \mathfrak{B}$,
\begin{align*}
\int_{E_Q}f(y)\,dy
&\leq
\sum_{\genfrac{}{}{0pt}{}{Q'\in\operatorname{ch}(Q)}{Q'\notin\mathfrak B}}
\int_{Q'}f(y)\,dy\leq
\sum_{Q'\in\operatorname{ch}(Q)}
\eta\lambda \ell(Q')^D
\\
&=
2^n \eta\lambda \left(\frac{\ell(Q)}2\right)^D=
2^{-\gamma}\eta\lambda \ell(Q)^D,
\end{align*}
which further implies that $\mu_Q$ satisfies \eqref{9}
for any $Q\in\mathfrak{B}$.
From \eqref{e4.2}, it follows that, for any $(x,t)\in \mathcal U^\complement$,
if \(T f(x,t)>\lambda\), then
$H(x,t)>\frac1{2^{-\gamma}\eta}.$
Choose \(\eta>0\) sufficiently small such that
$\beta:=\frac{1}{2^{-\gamma}\eta C}\ge 1,$
where the positive constant $C$ is the same as in \eqref{e5.1}.
Then Lemma \ref{lem-main}(i) yields that
\begin{align}\label{e6.1}
\nu_\gamma\bigl(\mathcal U^\complement\cap\{T f>\lambda\}\bigr)
&\lesssim
\eta |\Omega|_{L(\log L)(\mathbb{S}^{n-1})}
\sum_{Q\in\mathfrak B}\ell(Q)^D
\nonumber\\
&\lesssim
\eta |\Omega|_{L(\log L)(\mathbb{S}^{n-1})}
\frac{\|f\|_{L^1(\mathbb{R}^n)}}{\eta\lambda}
\lesssim
|\Omega|_{L(\log L)(\mathbb{S}^{n-1})}
\frac{\|f\|_{L^1(\mathbb{R}^n)}}{\lambda},
\end{align}
where the second step used
an argument similar to the estimation of \eqref{e3.1}.
Combining \eqref{e6.1} and \eqref{e3.1}, we have
\[
\nu_\gamma\{T f>\lambda\}
\lesssim\left[1+|\Omega|_{L(\log L)(\mathbb{S}^{n-1})}\right]
\frac{\|f\|_{L^1(\mathbb{R}^n)}}{\lambda},
\]
where the implicit positive constant depends only on $n$ and $\gamma$.
This then finishes the proof of Theorem \ref{thm-up-p=1}.
\end{proof}
Now, we are ready to prove the sufficiency of Theorem \ref{thm-diag}.
\begin{proof}[Proof of the sufficiency of Theorem \ref{thm-diag}]
The lower estimate follows from Proposition \ref{thm-lower}.
The upper estimate of \eqref{eq-p>1} is proved in \eqref{eq-up-p>1}.
The upper estimate of \eqref{eq-p=1} follows from Theorem \ref{thm-up-p=1}.
This finishes the proof of the sufficiency of Theorem \ref{thm-diag}.
\end{proof}

\subsection{Necessity of Theorem \ref{thm-diag}}\label{s2.4}
In this subsection, we show the necessity of Theorem \ref{thm-diag}; that is,
we prove the following conclusion.

\begin{theorem}\label{thm-sharpness}
Let \(n\in\mathbb N\cap[2,\infty)\), \(p\in[1,\infty)\), and
\(\Omega\in L^1(\mathbb S^{n-1})\) with \(\Omega\not\equiv0\). Assume that either
$p\in(1,\infty)\ \text{and}\ \gamma=0,$
or
$p=1\ \text{and}\  \gamma\in[-n,0].$
Then there exists a function \(f\in L^p(\mathbb R^n)\) such that
\[
\sup_{\lambda\in(0,\infty)}
\lambda^p
\underset{\mathcal{M}^\Omega(f)(x,t)>\lambda t^{\frac{\gamma}{p}}}
{\int_{\mathbb R^n}\int_0^\infty}
t^{\gamma-1}\,dt\,dx
=\infty,
\]
where $\mathcal{M}^\Omega$ is as in \eqref{eq-M}.
\end{theorem}

To prove Theorem \ref{thm-sharpness}, the key point is to
establish the lower estimate of the integrals of
the kernel $\Omega$ on some cubes or balls. To be
precise, we need the following lemmas.
\begin{lemma}\label{lem-directional-lower}
Let \(\Omega\in L^1(\mathbb S^{n-1})\) be nonnegative with \(\Omega\not\equiv0\). Then
there exist \(e\in\mathbb S^{n-1}\), \(A\in(1,\infty)\), and
\(c_0\in(0,\infty)\), depending only on $\Omega$, such that,
for any dyadic cube \(Q\subset\mathbb R^n\) and any \(y\in Q\),
\begin{align}\label{eq-directional-lower}
\frac1{|Q|}
\int_{Q+A\ell(Q)e}
\Omega\left(\frac{x-y}{|x-y|}\right)\,dx
\ge c_0.
\end{align}
\end{lemma}

\begin{proof}
Since $\Omega\equiv0$, from
the Lebesgue differentiation theorem
on $\mathbb{S}^{n-1}$ (see, for instance, \cite[Lemma 2.3]{dlyyz22}),
it follows that there exists a Lebesgue point \(e\in\mathbb S^{n-1}\)
of \(\Omega\) satisfying $\Omega(e)>0.$
Let $Q_0:=[0,1)^n.$
We first claim that
\begin{align}\label{eq-uniform-directional-conv}
\lim_{A\to\infty}
\sup_{y\in Q_0}
\int_{Q_0}
\left|
\Omega\left(\frac{Ae+x-y}{|Ae+x-y|}\right)-\Omega(e)
\right|
\,dx=0.
\end{align}
Indeed, fix \(A\ge 2\sqrt n\) and \(y\in Q_0\).
Take a nonnegative measurable function $H$ on $\mathbb{R}^n$.
Since $A\ge 2\sqrt{n}$,
we deduce that $|Ae+x-y|\ge A - \sqrt{n}>0$ for any $x\in Q_0$.
By this, a change of variables, and polar coordinates,
we find that
\begin{align}\label{eq-hu-1}
\int_{Q_0}H\left(\frac{Ae+x-y}{|Ae+x-y|}\right)\,dx
&=\int_{Ae+Q_0-y} H\left(\frac z{|z|}\right)\,dz\nonumber\\
&=\int_{\mathbb S^{n-1}}H(\xi)
\int_0^\infty
\mathbf 1_{Ae+Q_0-y}(r\xi)r^{n-1}\,dr\,d\sigma(\xi).
\end{align}
Note that, if \(r\xi\in Ae+Q_0-y\), then \(|r-A|\le |r\xi-Ae|\le\sqrt n\).
Moreover,
\begin{align*}
|\xi-e|=\left|\frac{r\xi}{|r\xi|}-\frac{Ae}{|Ae|} \right|
\le\left|\frac{r\xi}{|r\xi|}-\frac{r\xi}{|Ae|}\right|
+\left|\frac{r\xi}{|Ae|}-\frac{Ae}{|Ae|}\right|\le
\frac{2\sqrt{n}}{A}.
\end{align*}
These, together with \eqref{eq-hu-1}, further imply that
\begin{align}\label{eq-hu}
\int_{Q_0}H\left(\frac{Ae+x-y}{|Ae+x-y|}\right)\,dx\lesssim A^{n-1}
\int_{B_{\mathbb S^{n-1}}(e,\frac{2\sqrt{n}}{A})}H(\xi)\,d\sigma(\xi),
\end{align}
where the implicit positive constant is independent of \(A\) and \(y\)
and, for any $\xi\in\mathbb{S}^{n-1}$ and $r\in(0,\infty)$,
\[B_{\mathbb S^{n-1}}(\xi,r):=\{\eta\in\mathbb{S}^{n-1}:\ |\eta-\xi|<r\}.\]
Applying \eqref{eq-hu} with
\(
H(\xi):=|\Omega(\xi)-\Omega(e)|,
\)
we find that
\begin{align*}
\int_{Q_0}
\left|
\Omega\left(\frac{Ae+x-y}{|Ae+x-y|}\right)-\Omega(e)
\right|
\,dx
&\lesssim A^{n-1}
\int_{B_{\mathbb S^{n-1}}(e,\frac{2\sqrt{n}}{A})}
|\Omega(\xi)-\Omega(e)|\,d\sigma(\xi)
\\
&\lesssim
\frac{1}{\sigma(B_{\mathbb S^{n-1}}(e,\frac{2\sqrt{n}}{A}))}
\int_{B_{\mathbb S^{n-1}}(e,\frac{2\sqrt{n}}{A})}
|\Omega(\xi)-\Omega(e)|\,d\sigma(\xi),
\end{align*}
where the implicit positive constants are independent of $A$ and $y$.
Since \(e\) is a Lebesgue point of \(\Omega\),
it follows that the last term tends to \(0\) as
\(A\to\infty\). This proves the aformentioned claim \eqref{eq-uniform-directional-conv}.

Now, choose \(A\in(1,\infty)\) sufficiently large so that
\begin{align*}
\sup_{y\in Q_0}
\int_{Q_0}
\left|
\Omega\left(\frac{Ae+x-y}{|Ae+x-y|}\right)-\Omega(e)
\right|
\,dx\le \frac{\Omega(e)}{2}.
\end{align*} Then, for any \(y\in Q_0\),
\[
\int_{Q_0}
\Omega\left(\frac{Ae+x-y}{|Ae+x-y|}\right)\,dx
\ge \frac {\Omega(e)}2 .
\]
Finally, let \(Q=x_Q+\ell(Q)Q_0\) be an arbitrary dyadic cube and write
\[
y=x_Q+\ell(Q)y'\quad
\text{and}\quad
x=x_Q+\ell(Q)x'+A\ell(Q)e,
\]
where $x_Q$ denotes the lower-left corner of $Q$ and \(x',y'\in Q_0\).
Then
\[
\frac1{|Q|}
\int_{Q+A\ell(Q)e}
\Omega\left(\frac{x-y}{|x-y|}\right)\,dx
=
\int_{Q_0}
\Omega\left(\frac{Ae+x'-y'}{|Ae+x'-y'|}\right)\,dx'
\ge \frac {\Omega(e)}2.
\]
This proves \eqref{eq-directional-lower} with \(c_0:=\frac{\Omega(e)}{2}\), which then
completes the proof of Lemma \ref{lem-directional-lower}.
\end{proof}

\begin{lemma}\label{lem-good-set}
Let \(\Omega,e,\) and \(A\) be as in Lemma \ref{lem-directional-lower}. Then,
for any dyadic cube \(Q\subset\mathbb R^n\) and any
nonnegative measurable function \(h\) on $\mathbb{R}^{n}$, there exists a measurable set
$G_Q\subset Q+A\ell(Q)e$
such that
\begin{align}\label{eq-GQ-size}
|G_Q|\ge c_1|Q|
\end{align}
and, for any \(x\in G_Q\),
\begin{align}\label{eq-GQ-lower}
\int_Q h(y)\Omega\left(\frac{x-y}{|x-y|}\right)\,dy
\ge
c_1\int_Q h(y)\,dy,
\end{align}
where the positive constant $c_1$ depends only on $\Omega$.
\end{lemma}

\begin{proof}
Fix a dyadic cube \(Q\) and a nonnegative measurable
function $h$. Choose $N\in\mathbb{N}$ such that
$\Omega_N:=\min\{\Omega,N\}\not\equiv0$.
Let
$X_Q:=Q+A\ell(Q)e$ and, for any bounded set $E\subset\mathbb R^n$,
\begin{align}\label{eq-mq}
m_E(h):=\int_E h(y)\,dy.
\end{align}
If \(m_Q(h)=0\), the desired conclusion is trivial. Assume \(m_Q(h)>0\).
For any $x\in X_Q$, define
\[
F_Q(x):=
\int_Q h(y)\Omega_N\left(\frac{x-y}{|x-y|}\right)\,dy.
\]
By Lemma \ref{lem-directional-lower} and Tonelli's theorem,
we obtain
\begin{align}\label{eq-fl}
\int_{X_Q}F_Q(x)\,dx
=
\int_Q h(y)
\int_{X_Q}
\Omega_N\left(\frac{x-y}{|x-y|}\right)\,dx\,dy
\ge c_0 m_Q(h)|Q|,
\end{align}
where the positive constant $c_0$ is as in Lemma \ref{lem-directional-lower}.
On the other hand, since \(0\le\Omega_N\le N\), it follows that,
for any $x\in X_Q$,
\begin{align}\label{eq-fu}
0\le F_Q(x)\le N m_Q(h).
\end{align}
Let
\[
G_Q:=
\left\{x\in X_Q:F_Q(x)\ge \frac{c_0}{2}m_Q(h)\right\}.
\]
Then, using \eqref{eq-fu} and \eqref{eq-fl}, we conclude that
\[
c_0m_Q(h)|X_Q|
\le
\int_{X_Q}F_Q(x)\,dx
<
N m_Q(h)|G_Q|+\frac{c_0}{2}m_Q(h)|X_Q|.
\]
Thus,
$|G_Q|\ge \frac{c_0}{2N}|X_Q|=\frac{c_0}{2N}|Q|.$
Moreover, for any \(x\in G_Q\), since \(\Omega_N\le\Omega\), we infer that
\[
\int_Q h(y)\Omega\left(\frac{x-y}{|x-y|}\right)\,dy
\ge
F_Q(x)
\ge
\frac{c_0}{2}m_Q(h).
\]
Taking \(c_1:=\frac{c_0}{2N}\), we obtain
\eqref{eq-GQ-size} and \eqref{eq-GQ-lower}. This finishes the proof of Lemma
\ref{lem-good-set}.
\end{proof}

\begin{lemma}\label{lem-far-field}
Let \(\Omega\in L^1(\mathbb S^{n-1})\) be nonnegative with
\(\Omega\not\equiv0\), $R\in(0,\infty)$, and
\(h\) be nonnegative measureable on $\mathbb{R}^n$.
Then, for any \(\rho\ge4R\), there is
a measurable set \(E_\rho\subset\mathbb S^{n-1}\) satisfying that
\(\sigma(E_\rho)\ge c\) and, for any \(\xi\in E_\rho\),
\begin{align}\label{eq-farf}
\int_{B({\bf 0},R)}
h(y)\Omega\left(\frac{\rho\xi-y}{|\rho\xi-y|}\right)\,dy
\ge c\int_{B({\bf 0},R)}h(y)\,dy,
\end{align}
where the positive constant $c$
depends only on $n$ and $\Omega$.
\end{lemma}

\begin{proof}
Without loss of generality, we may assume \(m_{B_R}(h)>0\),
where $B_R:=B({\bf 0},R)$ and $m_{B_R}(h)$ is as in \eqref{eq-mq}.
Choose \(N\in\mathbb N\) such that
$\Omega_N:=\min\{\Omega,N\}\not\equiv0$. Fix \(\rho\ge4R\) and
\(y\in B({\bf 0},R)\), and define
\[
\Phi_{\rho,y}:\mathbb S^{n-1}\to\mathbb S^{n-1},
\qquad
\Phi_{\rho,y}(\xi):=
\frac{\rho\xi-y}{|\rho\xi-y|}.
\]
Clearly, for any $\xi\in\mathbb{S}^{n-1}$,
$\rho\xi\neq y$ and, moreover,
\(\Phi_{\rho,y}\) is a diffeomorphism, by which we mean that it is a
one-to-one \(C^1\) mapping whose inverse is also \(C^1\). To see this, if
\(\eta=\Phi_{\rho,y}(\xi)\), then there exists a unique \(s>0\) such that
$\rho\xi-y=s\eta$, which yields that
\[
s(\eta)=
-y\cdot\eta+
\left[\rho^2-|y|^2+(y\cdot\eta)^2\right]^{\frac12},
\]
and hence $
\xi=\frac{y+s(\eta)\eta}{\rho}$
gives the inverse mapping of $\Phi_{\rho,y}$, which is a $C^1$ mapping clearly.
A basic computation of the Jacobian $J_{\rho,y}$ of the mapping $\Phi_{\rho,y}$
yields that, for any $\xi\in\mathbb S^{n-1}$,
\[
J_{\rho,y}(\xi)
=\left(\frac{\rho}{|\rho\xi-y|}\right)^{n-1}
\frac{|\rho-\xi\cdot y|}{|\rho\xi-y|}.
\]
Since \(|y|\le\frac\rho 4\), we deduce that
$\frac{3\rho}{4}\le|\rho\xi-y|\le\frac{5\rho}{4}$ and
$\rho-\xi\cdot y\ge\frac{3\rho}{4}.$
Therefore, $J_{\rho,y}(\xi)\sim 1,$
where the positive equivalence constants
depend only on $n$. This, combined with the
change of variables \(\eta=\Phi_{\rho,y}(\xi)\) and
Tonelli's theorem, further implies that
\begin{align}\label{eq-imh}
&\int_{\mathbb S^{n-1}}\int_{B({\bf 0},R)}
h(y)\Omega_N\left(\frac{\rho\xi-y}{|\rho\xi-y|}\right)\,dy\,d\sigma(\xi)\nonumber\\
&\quad\ge \int_{B({\bf 0},R)}h(y)\int_{\mathbb S^{n-1}}
\Omega_N(\Phi_{\rho,y}(\xi))\,d\sigma(\xi)\,dy\nonumber\\
&\quad=\int_{B({\bf 0},R)}h(y)\int_{\mathbb S^{n-1}}
\frac{\Omega_N(\eta)}{J_{\rho,y}(\Phi_{\rho,y}^{-1}(\eta))}
\,d\sigma(\eta)\,dy\nonumber\\
&\quad\ge c_n\|\Omega_N\|_{L^1(\mathbb{S}^{n-1})}m_{B_R}(h),
\end{align}
where the positive constant $c_n$ depends only on $n$.
Write
\(
A_\Omega:=c_n\|\Omega_N\|_{L^1(\mathbb S^{n-1})}>0
\)
and define
\[
E_\rho:=
\left\{
\xi\in\mathbb S^{n-1}:
\int_{B({\bf 0},R)}
h(y)\Omega_N\left(\frac{\rho\xi-y}{|\rho\xi-y|}\right)\,dy\ge
\frac{A_\Omega}{2\sigma(\mathbb S^{n-1})}m_{B_R}(h)
\right\}.
\]
Since $\Omega_N\le N$, from \eqref{eq-imh} and the definition of
$E_\rho$, it follows that
\[
\frac{A_\Omega}{2}m_{B_R}(h)
\le
\int_{E_\rho}I_\rho(\xi)\,d\sigma(\xi)
\le
N\sigma(E_\rho)m_{B_R}(h)
\]
and hence $\sigma(E_\rho)\ge\frac{A_\Omega}{2N}.$
Taking
$c:=\min\{
\frac{\sigma(\mathbb S^{n-1})}{2},
\frac{A_\Omega}{2N},
\frac{A_\Omega}{2\sigma(\mathbb S^{n-1})}\}$,
we obtain \eqref{eq-farf}.
This finishes the proof of Lemma \ref{lem-far-field}.
\end{proof}

Now, we are ready to prove Theorem \ref{thm-sharpness}.
\begin{proof}[Proof of Theorem \ref{thm-sharpness}]
The proof mainly follows along the same lines of the one of \cite[Theorem 2.6]{dlyyz26}.
The key difference is to estimate the lower bound of the kernel $\Omega$ by
Lemmas \ref{lem-good-set} and \ref{lem-far-field}. Thus, we only sketch the proof
and highlight the main differences from the proof of \cite[Theorem 2.6]{dlyyz26}.
Without loss of generality, we may assume that $\Omega$ is nonnegative.

Now, we consider the following three cases.

\emph{Case 1: \(p\in[1,\infty)\) and \(\gamma=0\).}
Let $f:=\mathbf 1_{B({\bf 0},2)}.$
If \(x\in B({\bf 0},1)\) and \(t\in(0,1]\), then \(B(x,t)\subset B({\bf 0},2)\). Therefore,
\[
\begin{aligned}
\mathcal{M}^\Omega(f)(x,t)
&\ge
\frac1{|B(x,t)|}
\int_{B(x,t)}
\Omega\left(\frac{x-y}{|x-y|}\right)\,dy
\\
&=
\frac1{|B({\bf 0},1)|}
\int_{B({\bf 0},1)}
\Omega\left(\frac z{|z|}\right)\,dz=:c_\Omega>0.
\end{aligned}
\]
Thus, for any \(\lambda\in(0,c_\Omega)\),
\[
B({\bf 0},1)\times(0,1]
\subset
\left\{(x,t)\in\mathbb R^{n+1}_+:\mathcal{M}^\Omega(f)(x,t)>\lambda\right\},
\]
which further implies that
\[
\lambda^p
\underset{\mathcal{M}^\Omega(f)(x,t)>\lambda}
{\int_{\mathbb R^n}\int_0^\infty}
\frac{dt}{t}\,dx
\gtrsim
\lambda^p\int_0^1\frac{dt}{t}
=\infty .
\]
This proves the desired conclusion in this case.

\emph{Case 2: \(p=1\) and \(\gamma=-n\).}
Again let $f:=\mathbf 1_{B({\bf 0},2)}.$
Applying Lemma \ref{lem-far-field} with \(R=2\) and \(h=f\),
we conclude that, for any \(\rho\ge 8\),  there exists
$E_\rho\subset \mathbb{S}^{n-1}$ such that, for any $\xi\in E_\rho$,
\[
\int_{B({\bf 0},2)}
\Omega\left(\frac{\rho\xi-y}{|\rho\xi-y|}\right)\,dy
\ge c|B({\bf 0},2)|
\]
and \(\sigma(E_\rho)\ge c\), where the positive constant $c$ depends only on
$n$ and $\Omega$.
For any \(m\in\mathbb N\) with \(m\ge 3\), define
$X_m:=\{\rho\xi:2^m<\rho\le2^{m+1},\ \xi\in E_\rho\}$
and
\(
r_m:=2^{m+2}.
\)
Then, for any \(x\in X_m\) and \(t\in(\frac{r_m}{2},r_m]\),
$B({\bf 0},2)\subset B(x,r_m)$ and
\begin{align*}
\mathcal{M}^\Omega(f)(x,t)
\ge
\frac1{|B(x,r_m)|}
\int_{B({\bf 0},2)}
\Omega\left(\frac{x-y}{|x-y|}\right)\,dy
\gtrsim t^{-n}.
\end{align*}
This implies that there exists a positive constant $\lambda_0$,
depending only on $n$ and $\Omega$, such that
\[
\bigcup_{m=3}^{\infty}X_m\times\left(\frac{r_m}{2},r_m\right]
\subset
\left\{(x,t)\in\mathbb R^{n+1}_+:
\mathcal{M}^\Omega(f)(x,t)>\lambda_0t^{-n}\right\}.
\]
Consequently,
\begin{align*}
\underset{\mathcal{M}^\Omega(f)(x,t)>\lambda_0t^{-n}}
{\int_{\mathbb R^n}\int_0^\infty}
t^{-n-1}\,dt\,dx
&
\ge
\sum_{m=3}^\infty
|X_m|
\int_{\frac{r_m}{2}}^{r_m}t^{-n-1}\,dt
\\
&
\gtrsim
\sum_{m=3}^\infty
2^{mn}2^{-mn}
=
\infty .
\end{align*}
This proves the desired conclusion when \(p=1\) and \(\gamma=-n\).

\emph{Case 3: \(p=1\) and \(\gamma\in(-n,0)\).}
Let
\(
D:=n+\gamma\in(0,n).
\)
Choose an increasing sequence \(\{n_k\}_{k=1}^\infty\subset\mathbb N\) such that,
for any $k\in\mathbb{N}$,
\(
Dn_k\ge k.
\)
For each \(k\in\mathbb N\), let \(q_k\in\mathbb N\) be chosen so that
\begin{align*}
2^{Dn_k}-1<q_k\le 2^{Dn_k}.
\end{align*}
Since \(D<n\), one has \(q_k< 2^{nn_k}\).

Let
\(
Q_0:=[0,1)^n.
\)
We construct a sequence of finite collections of dyadic cubes
\(\{\mathcal C_k\}_{k=1}^\infty\) as follows. Let \(\mathcal C_1\) consist of \(q_1\)
dyadic subcubes of \(Q_0\) with edge length \(2^{-n_1}\). Suppose that \(\mathcal C_k\)
has been constructed. For each \(Q\in\mathcal C_k\), choose \(q_{k+1}\) dyadic
subcubes of \(Q\) with edge length \(2^{-n_{k+1}}\ell(Q)\), and denote the resulting
collection by \(\mathcal C_{k+1}(Q)\). Let
\[
\mathcal C_{k+1}:=\bigcup_{Q\in\mathcal C_k}\mathcal C_{k+1}(Q).
\]
Then the cubes in each \(\mathcal C_k\)
are pairwise disjoint,
\begin{align}\label{eq-Ck-data}
\sharp\mathcal C_k=q_1\cdots q_k,
\quad\text{and}\quad
\ell(Q)=2^{-(n_1+\cdots+n_k)}
\quad\text{for all }Q\in\mathcal C_k.
\end{align}
For simplicity, write
\(
S_k:=n_1+\cdots+n_k.
\)
Define
\begin{align}\label{eq-counterexample-f}
f:=
\sum_{k=1}^\infty
\frac{2^{-\gamma S_k}}{k^{\frac{3}{2}}}
\sum_{Q\in\mathcal C_k}\mathbf 1_Q.
\end{align}
From the estimates in \cite[p.\,19]{dlyyz26},
we infer \(f\in L^1(\mathbb R^n)\).

Now, fix \(k\in\mathbb N\) and \(Q\in\mathcal C_k\). For
any \(i\in\mathbb N\), write
\(
\mathcal C_{k+i}(Q):=
\{P\in\mathcal C_{k+i}:P\subset Q\}.
\)
Define
\[
h_Q:=
\sum_{i=1}^\infty
\frac{2^{-\gamma S_{k+i}}}{(k+i)^{\frac{3}{2}}}
\sum_{P\in\mathcal C_{k+i}(Q)}\mathbf 1_P.
\]
Note that \(0\le h_Q\le f\mathbf 1_Q\) and $h_Q$ is precisely
the right-hand side of \cite[(2.20)]{dlyyz26}. Thus, using
\cite[(2.26)]{dlyyz26}, we find that
\begin{align*}
\fint_Q h_Q(y)\,dy
\gtrsim
k^{-\frac{1}{2}}\ell(Q)^\gamma .
\end{align*}
From this and Lemma \ref{lem-good-set} with \(h=h_Q\),
we infer a measurable set
$G_Q\subset Q+A\ell(Q)e$
such that
\(
|G_Q|\gtrsim |Q|
\)
and, for any \(x\in G_Q\),
\begin{align}\label{eq-rough-tail-lower}
\fint_Q h_Q(y)\Omega\left(\frac{x-y}{|x-y|}\right)\,dy
\gtrsim
k^{-\frac{1}{2}}\ell(Q)^\gamma,
\end{align}
where $A$ and $e$ are as in Lemma \ref{lem-directional-lower}.
Let $R_Q:=(A+2\sqrt n)\ell(Q)$ and $I_Q:=(\frac{R_Q}{2},R_Q].$
Note that, if \(x\in G_Q\), then \(Q\subset B(x,R_Q)\).
Using this,
\eqref{eq-counterexample-f}, and \eqref{eq-rough-tail-lower},
we conclude that, for any $x\in G_Q$ and $t\in I_Q$,
\begin{align*}
\mathcal{M}^\Omega(f)(x,t)
&\ge
\frac1{|B(x,R_Q)|}
\int_Q
h_Q(y)\Omega\left(\frac{x-y}{|x-y|}\right)\,dy
\gtrsim
R_Q^{-n}k^{-\frac{1}{2}}|Q|\ell(Q)^\gamma
\gtrsim
k^{-\frac{1}{2}}\ell(Q)^\gamma .
\end{align*}
Since \(t\sim\ell(Q)\) for \(t\in I_Q\), we deduce that there exists a
positive constant
\(c_3\), depending only on $n,\gamma$, and $\Omega$, such that,
for any \(x\in G_Q\) and \(t\in I_Q\),
\begin{align}\label{eq-max-lower-cantor}
\mathcal{M}^\Omega(f)(x,t)\ge c_3 k^{-\frac{1}{2}}t^\gamma.
\end{align}
For any \(N\in\mathbb N\), let
\(
\lambda_N:=\frac{c_3}{2}N^{-\frac{1}{2}}.
\)
Then, by \eqref{eq-max-lower-cantor}, for any \(1\le k\le N\) and
\(Q\in\mathcal C_k\),
\[
G_Q\times I_Q
\subset
\left\{(x,t)\in\mathbb R^{n+1}_+:\mathcal{M}^\Omega(f)(x,t)>\lambda_N t^\gamma\right\}.
\]
Note that the sets \(G_Q\times I_Q\), with \(1\le k\le N\) and \(Q\in\mathcal C_k\), are pairwise
disjoint.
Thus, using \eqref{eq-Ck-data} and \cite[(2.25)]{dlyyz26}, we find that, for any $N\in\mathbb N$,
\begin{align*}
&\lambda_N
\underset{\mathcal{M}^\Omega(f)(x,t)>\lambda_Nt^\gamma}
{\int_{\mathbb R^n}\int_0^\infty}
t^{\gamma-1}\,dt\,dx
\\
&\quad
\ge
\lambda_N
\sum_{k=1}^N
\sum_{Q\in\mathcal C_k}
|G_Q|
\int_{I_Q}t^{\gamma-1}\,dt
\gtrsim
N^{-\frac{1}{2}}
\sum_{k=1}^N
\sum_{Q\in\mathcal C_k}
|Q|\ell(Q)^\gamma
\\
&\quad
=
N^{-\frac{1}{2}}
\sum_{k=1}^N
\prod_{j=1}^k\frac{q_j}{2^{Dn_j}}
\gtrsim
N^{-\frac{1}{2}}
\sum_{k=1}^N 1
\gtrsim
N^{\frac{1}{2}}.
\end{align*}
Letting \(N\to\infty\), we obtain the desired conclusion when \(p=1\) and \(\gamma\in(-n,0)\).

Combining the above three cases
then completes the proof of
Theorem \ref{thm-sharpness}.
\end{proof}

\section{Applications}\label{s3}
In this section, we provide several
applications of Theorem \ref{thm-diag} by three subsections.
In Subsection \ref{s3.1},
we obtain weak type $(1,1)$ estimates for some generalized Poisson integrals
whose kernels do not need additional logarithmic integrability.
Subsection \ref{s3.2} is devoted to strengthening the upper estimate in
Theorem \ref{thm-up-p=1} by the lifted operator $\mathcal M^\ast_\Omega$.
In Subsection \ref{s3.3}, we establish a new weak-type representation for
truncated rough singular integrals, which further gives a characterization
of Hardy spaces.

\subsection{Weak Type $(1,1)$ Estimates for Generalized Poisson Integrals}\label{s3.1}

In this subsection, we aim to answer the question posed by
Sj\"ogren and Soria \cite[p.\,228]{ss96} regarding the weak-type estimate \eqref{eq-io}
of generalized Poisson integrals and extend it to
the higher-dimensional case.

Precisely speaking,
we aim to determine whether the estimate \eqref{eq-io}
holds for $\alpha<0$ and general $h$ without
additional logarithmic integrability. Indeed,
by the change of variables $y=t^{-1}$ and $\gamma=-\alpha > 0$,
the estimate \eqref{eq-io} becomes,
with the implicit positive constant independent
of $f$ and $h$,
\begin{align}\label{eq-im}
\sup_{\lambda\in(0,\infty)}\lambda
\int_{0}^{\infty}\left|\left\{x\in\mathbb{R}:
t^{-(1+\gamma)}\int_{\mathbb R}f(y)h\left(\frac{x-y}{t}\right)>\lambda \right\} \right|
t^{\gamma-1}\,dt\,dx
\lesssim \|h\|_{L^1(\mathbb R)}\|f\|_{L^1(\mathbb R)},
\end{align}
which is a one-dimensional case of the following conclusion.
\begin{theorem}\label{thm:rough_application}
Let $n\in\mathbb N\cap[2,\infty)$, $\gamma\in(0,\infty)$, and
$\Omega\in L(\log L)(\mathbb{S}^{n-1})$ be
a nonnegative function.
Suppose that, for any $x \in \mathbb{R}^n \setminus \{{\bf 0}\}$,
\begin{equation*}
h(x):= \Omega\left(\frac{x}{|x|}\right) \phi(|x|),
\end{equation*}
where $\phi: [0, \infty) \to [0, \infty)$ is a nonincreasing
$L^1$-function. Define the integral operator $T_h$ by setting,
for any $f \in L^1_{\rm loc}(\mathbb{R}^n)$ and $(x,t) \in \mathbb{R}_{+}^{n+1}$,
\begin{equation*}
T_h(f)(x,t) := t^{-(n+\gamma)} \int_{\mathbb{R}^n} f(y)
h\left(\frac{x-y}{t} \right)\,dy.
\end{equation*}
Then there exists a positive constant $C$ such that, for any $f\in L^1(\mathbb R^{n})$,
\begin{align}\label{eq:weak_bound}
&\sup_{\lambda \in(0,\infty)} \lambda\nu_\gamma\left(\left\{ (x,t) \in \mathbb{R}_{+}^{n+1} :
|T_h(f)(x,t)| > \lambda \right\}\right) \nonumber\\
&\quad\le C \|\phi(|\cdot|)\|_{L^1(\mathbb{R}^n)} \|f\|_{L^1(\mathbb{R}^n)},
\end{align}
where $\nu_\gamma$ is as in \eqref{dv} and
the constant $C$ depends only on $n$, $\gamma$, and $|\Omega|_{L(\log L)(\mathbb S^{n-1})}$.
Moreover, \eqref{eq:weak_bound} also holds for $n=1$ and $\Omega\equiv 1$.
\end{theorem}

\begin{remark}\label{rem:open_question}
Theorem \ref{thm:rough_application} when $n=1$ and $\Omega\equiv1$
implies \eqref{eq-im} and hence \eqref{eq-io}. Thus, it gives an
affirmative answer to the question posed in \cite[p.\,228]{ss96}.
Moreover, Theorem \ref{thm:rough_application} also provides a higher-dimensional
generalization of the aformentioned question.
\end{remark}

To prove Theorem \ref{thm:rough_application},
we need the following lemma; see \cite[Lemma 2.3]{snw69}.
\begin{lemma}[Stein--N. Weiss Adding-Up Lemma]\label{lem:stein_weiss}
Let $(X, \mu)$ be a measure space and $\{g_j\}_{j=1}^\infty$ be a sequence of
nonnegative measurable functions on $X$ satisfying that, for any $j \in \mathbb{N}$,
\begin{equation*}
\sup_{s\in(0,\infty)}s\mu(\{x \in X : g_j(x) > s\}) \le 1.
\end{equation*}
Let $\{c_j\}_{j=1}^\infty$ be a sequence of positive numbers
such that $\sum_{j=1}^\infty c_j = 1$, and define the entropy constant
\begin{equation*}
K := \sum_{j=1}^\infty c_j \log\left(\frac{1}{c_j}\right).
\end{equation*}
Then, for any $s \in (0,\infty)$,
\begin{equation*}
\mu\left(\left\{x \in X : \sum_{j=1}^\infty c_j g_j(x) > s \right\}\right) \le \frac{2(K+2)}{s}.
\end{equation*}
\end{lemma}

Next, we show Theorem \ref{thm:rough_application}.
\begin{proof}[Proof of Theorem \ref{thm:rough_application}]
Without loss of generality, we may assume that $f \ge 0$ and
$\|\phi(|\cdot|)\|_{L^1(\mathbb{R}^n)} = 1$.
By a standard limiting argument, we may also assume that $\mathrm{supp}(\phi) \subset [0, 1]$.

First, we define the linear operator $A^\Omega$
associated with the rough kernel $\Omega$ at the scale $t\in(0,\infty)$ by setting, for any
$f\in L^1_{\rm loc}(\mathbb R^n)$ and $(x,t)\in\mathbb R^{n+1}_+$,
\begin{equation*}
A^{\Omega}(f)(x,t) := t^{-(n+\gamma)} \int_{|x-y| < t}
f(y) \Omega\left(\frac{x-y}{|x-y|}\right) dy.
\end{equation*}
Observe that, for any $(x,t) \in \mathbb{R}_{+}^{n+1}$,
$A^{\Omega}(f)(x,t) \le t^{-\gamma} \mathcal{M}^\Omega(f)(x,t)$, where $\mathcal{M}^\Omega$ is
as in \eqref{eq-M}. Since $\gamma > 0$, from Theorem \ref{thm-diag}(ii),
it follows that
\begin{equation}\label{eq:base_weak_bound}
\left\|A^{\Omega}(f)\right\|_{L^{1,\infty}(\mathbb{R}_{+}^{n+1},\nu_\gamma)} \lesssim
\|f\|_{L^1(\mathbb{R}^n)}.
\end{equation}

Define the radius function $R(\lambda) := \sup(\{r \ge 0 : \phi(r) > \lambda\}\cup\{0\})$.
Clearly, for any $\lambda\in(0,\infty)$, $R(\lambda)\in[0,1]$
because $\operatorname{supp}(\phi)\subset [0,1]$.
By the assumption that $\phi$ is nonnegative and nonincreasing,
we find that, for almost every $x\in\mathbb R^n$,
\begin{equation}\label{eq-px}
\phi(|x|) = \int_0^\infty {\bf 1}_{\{y\in\mathbb{R}^n : \phi(|y|) > \lambda\}}(x) \,d\lambda
= \int_0^\infty {\bf 1}_{B(0, R(\lambda))}(x) \,d\lambda,
\end{equation}
where $d\lambda$ denotes the standard one-dimensional Lebesgue measure
and $B(0,R(\lambda)):=\emptyset$ if $R(\lambda)=0$.
Using Tonelli's theorem, we conclude that
\begin{equation}\label{eq-1r}
1=\|\phi(|\cdot|)\|_{L^1(\mathbb{R}^n)} = \int_0^\infty |B(0, R(\lambda))| \,d\lambda
= v_n \int_0^\infty R(\lambda)^n \,d\lambda,
\end{equation}
where $v_n$ is the volume of the unit ball in $\mathbb{R}^n$.
Substituting \eqref{eq-px} into the definition of $T_h(f)$ and applying Tonelli's theorem,
we find that, for any $(x,t)\in\mathbb R^{n+1}_+$,
\begin{align}\label{eq-Th}
T_h(f)(x,t)
&= t^{-(n+\gamma)} \int_0^\infty  \int_{|x-y| < R(\lambda)t} f(y) \Omega
\left(\frac{x-y}{|x-y|}\right)\, dy\, d\lambda \nonumber\\
&= \int_0^\infty R(\lambda)^{n+\gamma} A^{\Omega}(f)(x, R(\lambda)t) \,d\lambda.
\end{align}
Now, for any $k\in\mathbb Z_+$, define
$E_k := \{\lambda \in(0,\infty) : 2^{-k-1} < R(\lambda) \le 2^{-k}\}.$
For any $\lambda \in E_k$, we have
\begin{align*}
A^{\Omega}(f)(x, R(\lambda)t)
&\le (2^{-k-1}t)^{-(n+\gamma)} \int_{|x-y| < 2^{-k}t} f(y)
\Omega\left(\frac{x-y}{|x-y|}\right) dy \\
&= 2^{n+\gamma} A^{\Omega}(f)(x, 2^{-k}t).
\end{align*}
From this and \eqref{eq-Th}, we infer that, for any $(x,t)\in\mathbb R^{n+1}_+$,
\begin{equation}\label{eq-ta}
T_h(f)(x,t) \le 2^{n+\gamma} \sum_{k=0}^\infty c_k A^{\Omega}(f)(x, 2^{-k}t),
\end{equation}
where $c_k := \int_{E_k} R(\lambda)^{n+\gamma} \,d\lambda$.
To apply the Stein--N. Weiss adding-up lemma, we verify
the entropy condition $\sum_{k=0}^\infty c_k \log(1/c_k) < \infty$.
Indeed, for any $k\in\mathbb Z_+$,
define $x_k := \int_{E_k} R(\lambda)^n \,d\lambda$.
Using \eqref{eq-1r}, we find that $\sum_{k=0}^\infty x_k = v_n^{-1} \le 1$ and, for any $k\in\mathbb Z_+$,
\begin{equation*}
c_k = \int_{E_k} R(\lambda)^\gamma R(\lambda)^n \,d\lambda
\le 2^{-k\gamma} \int_{E_k} R(\lambda)^n \,d\lambda = 2^{-k\gamma} x_k.
\end{equation*}
Using this estimate, the entropy condition is bounded by
\begin{align*}
\sum_{k=0}^\infty c_k \log\left(\frac{1}{c_k}\right)
&\le \sum_{k=0}^\infty 2^{-k\gamma} x_k \log\left(\frac{1}{2^{(-k-1)\gamma} x_k}\right) \\
&\sim \gamma \sum_{k=0}^\infty (k+1) 2^{-k\gamma} x_k + \sum_{k=0}^\infty 2^{-k\gamma}
\left( x_k \log\frac{1}{x_k} \right)<\infty,
\end{align*}
where the last step used $\gamma > 0$.
Therefore, the entropy condition is satisfied, which, combined with Lemma \ref{lem:stein_weiss},
\eqref{eq-ta}, and \eqref{eq:base_weak_bound}, further implies the desired conclusion.

Finally, by applying \eqref{eq-lifted-hl-weak} instead of Theorem \ref{thm-diag}(ii)
in \eqref{eq:base_weak_bound}, we easily find that
the present theorem also holds for $n=1$ and $\Omega \equiv 1$.
This finishes the proof of Theorem \ref{thm:rough_application}.
\end{proof}

\subsection{Weak Type $(1,1)$ Estimates for Lifted
Operator $\mathcal{M}^\ast_{\Omega}$}\label{s3.2}
In this subsection, our target is to
introduce a lifted version of the operator $M^\ast_\Omega$
defined in \eqref{eq-Ms} and establish its weak type $(1,1)$ estimates.

Given $\omega\in\mathbb S^{n-1}$, the \emph{lifted directional maximal operator
$\mathcal{M}_{\omega}$} is defined by setting, for any
$f \in L_{\rm loc}^1(\mathbb{R}^n)$ and $(x,t) \in \mathbb{R}_{+}^{n+1}$,
\begin{equation*}
\mathcal{M}_{\omega}(f)(x,t) := \sup_{r\in[t,2t)} \frac{1}{2r} \int_{-r}^r |f(x-s\omega)| \,ds.
\end{equation*}
Given $\Omega\in L^1(\mathbb S^{n-1})$,
the \emph{lifted operator $\mathcal M^\ast_\Omega$} is defined
by setting, for any $f \in L_{\rm loc}^1(\mathbb{R}^n)$ and $(x,t) \in \mathbb{R}_{+}^{n+1}$,
\begin{equation}\label{eq-lMw}
\mathcal{M}_{\Omega}^\ast(f)(x,t) := \int_{\mathbb{S}^{n-1}} |\Omega(\omega)|
\mathcal{M}_{\omega}(f)(x,t) \,d\sigma(\omega).
\end{equation}
Using polar coordinates and Tonelli's theorem, we easily find that
the lifted rough maximal operator $\mathcal{M}^\Omega$
is dominated by the operator $\mathcal{M}^\ast_\Omega$; that is,
for any $f\in L^1_{\rm loc}(\mathbb R^n)$ and $(x,t) \in \mathbb{R}_{+}^{n+1}$,
\begin{align*}
\mathcal{M}^\Omega(f)(x,t)\le \mathcal{M}^\ast_\Omega(f)(x,t),
\end{align*}
where $\mathcal{M}^\Omega$ is as in \eqref{eq-M}.

For the lifted
operator $\mathcal{M}_{\Omega}^\ast$, we can establish its weak type
$(1,1)$ estimates as follows.
\begin{theorem}\label{thm-stronger-2.5}
Let $n\in\mathbb{N}\cap[2,\infty)$,
$\gamma\in(-\infty,-n)\cup(0,\infty)$,
and $\Omega\in L(\log L)(\mathbb{S}^{n-1})$.
Then there exists a positive constant $C$, depending only on $n$ and $\gamma$,
such that, for any $f\in L^1(\mathbb{R}^n)$,
\eqref{eq-strong-up-p=1} holds.
\end{theorem}

\begin{proof}
Fix $f\in L^1(\mathbb{R}^n)$. Without loss of generality,
we may assume that $f$ is nonnegative.
For any given $(x,t)\in\mathbb{R}_{+}^{n+1}$, by the definitions
of $\mathcal{M}_\omega$ and $\mathcal{M}_{\Omega}^\ast$, we find that
\begin{align*}
\mathcal{M}_{\Omega}^\ast(f)(x,t)
&\le \frac{1}{2t}\int_{\mathbb{S}^{n-1}}|\Omega(\omega)|
\int_{-2t}^{2t}f(x-s\omega)\,ds\,d\sigma(\omega) \\
&= \frac{1}{2t}\int_{\mathbb{S}^{n-1}}|\Omega(\omega)|
\left[\int_{0}^{2t}f(x-s\omega)\,ds + \int_{0}^{2t}
f(x+s\omega)\,ds\right]\,d\sigma(\omega).
\end{align*}
Using the change of variables $\omega \mapsto -\omega$
on $\mathbb{S}^{n-1}$ for the second integral,
we obtain
\begin{equation*}
\int_{\mathbb{S}^{n-1}}|\Omega(\omega)|\int_{0}^{2t}f(x+s\omega)\,ds\,d\sigma(\omega)
= \int_{\mathbb{S}^{n-1}}|\Omega(-\omega)|\int_{0}^{2t}f(x-s\omega)\,ds\,d\sigma(\omega).
\end{equation*}
Define the symmetrized kernel $\widetilde{\Omega}(\omega) :=
\frac{1}{2}[|\Omega(\omega)| + |\Omega(-\omega)|]$.
Then the upper bound can be rewritten as
\begin{equation*}
\mathcal{M}_{\Omega}^\ast(f)(x,t) \le \frac{1}{t}\int_{\mathbb{S}^{n-1}}
\widetilde{\Omega}(\omega)\int_{0}^{2t}f(x-s\omega)\,ds\,d\sigma(\omega).
\end{equation*}
We decompose the radial integral over the interval $(0,2t]$ into a
disjoint union of dyadic intervals
$\bigcup_{j=0}^{\infty}(2^{-j}t, 2^{-j+1}t]$. Then
\begin{align}\label{eq-mms}
\mathcal{M}_{\Omega}^\ast(f)(x,t)
&\le \frac{1}{t}\sum_{j=0}^{\infty}\int_{\mathbb{S}^{n-1}}
\int_{2^{-j}t}^{2^{-j+1}t}f(x-s\omega)
\widetilde{\Omega}(\omega)\frac{s^{n-1}}{s^{n-1}}\,ds\,d\sigma(\omega) \nonumber\\
&\le \frac{1}{t}\sum_{j=0}^{\infty}\frac{1}{(2^{-j}t)^{n-1}}
\int_{\mathbb{S}^{n-1}}\int_{0}^{2^{-j+1}t}f(x-s\omega)
\widetilde{\Omega}(\omega)s^{n-1}\,ds\,d\sigma(\omega).
\end{align}
Using polar coordinates, we rewrite the inner integral over $\mathbb{R}^n$ as
\begin{equation*}
\int_{\mathbb{S}^{n-1}}\int_{0}^{2^{-j+1}t}f(x-s\omega)\widetilde{\Omega}(\omega)s^{n-1}
\,ds\,d\sigma(\omega) = \int_{B(x,2^{-j+1}t)}f(y)\left|
\widetilde{\Omega}\left(\frac{x-y}{|x-y|}\right)\right|dy.
\end{equation*}
Substituting this into \eqref{eq-mms}, we obtain
\begin{align*}
\mathcal{M}_{\Omega}^\ast(f)(x,t)
&\lesssim\sum_{j=0}^{\infty}2^{-j}\fint_{B(x,2^{-j+1}t)}f(y)
\left|\widetilde{\Omega}\left(\frac{x-y}{|x-y|}\right)\right|dy\\
&\le \sum_{j=0}^{\infty}2^{-j}\mathcal{M}^{\widetilde{\Omega}}(f)(x,2^{-j+1}t).
\end{align*}
Applying this, the Aoki--Rolewicz theorem,
$\widetilde{\Omega} \in L(\log L)(\mathbb{S}^{n-1})$ with
$|\widetilde{\Omega}|_{L(\log L)} \le |\Omega|_{L(\log L)}$, and Theorem \ref{thm-diag}(ii)
, we conclude that
there exists $\varrho\in(0,1)$ such that
\begin{align*}
\left\|\frac{\mathcal{M}_{\Omega}^\ast(f)(x,t)}{t^\gamma}\right\|_{L^{1,\infty}(\mathbb R^{n+1}_+,
\nu_\gamma)}^\varrho
&\le\sum_{j=0}^{\infty}2^{-j\varrho}
\left\|\frac{\mathcal{M}^{\widetilde{\Omega}}
(f)(x,2^{-j+1}t)}{t^\gamma}\right\|_{L^{1,\infty}(\mathbb R^{n+1}_+,
\nu_\gamma)}^\varrho\\
&\lesssim\sum_{j=0}^{\infty}2^{-j\varrho}\|f\|_{L^1(\mathbb R^n)}^\varrho\sim
\|f\|_{L^1(\mathbb R^n)}^\varrho.
\end{align*}
This finishes the proof of Theorem \ref{thm-stronger-2.5}.
\end{proof}

\subsection{$H^1$ Characterization in Terms of Truncated Rough Singular Integrals}
\label{s3.3}

In this subsection, we give an application of Theorem \ref{thm-diag}
to establish a weak-type representation of truncated rough
singular integrals  and a new characterization of Hardy spaces
in terms of them.

Let $n\in\mathbb N\cap[2,\infty)$ and $\Omega\in L^1(\mathbb S^{n-1})$.
For any $t\in(0,\infty)$ and $f\in \bigcup_{p\in[1,\infty)}L^p(\mathbb{R}^n)$,
the \emph{truncated rough singular integral $T_{\Omega,t}$} of $f$ is defined by setting,
for any $x\in\mathbb{R}^n$,
$$T_{\Omega,t}(f)(x)
:=\int_{|x-y|>t}K_\Omega(x-y)f(y)\,dy,$$
where $K_{\Omega}(x):=\frac{\Omega(x/|x|)}{|x|^n}$.
We assume that $\Omega$ satisfies the cancellation condition
\begin{equation}\label{5.1}
\int_{\mathbb S^{n-1}}\Omega(\theta)\,d\sigma(\theta)=0.
\end{equation}
Then the \emph{rough singular integral $T_\Omega$} is defined by setting,
for any $f\in C_{\rm c}^\infty(\mathbb R^n)$,
$$
T_{\Omega}(f)(\cdot):=\operatorname*{p.\!v.}
\int_{\mathbb R^n}K_\Omega(\cdot-y)f(y)\,dy.
$$
Here the \emph{symbol $C_{\rm c}^\infty(\mathbb R^n)$} denotes the set of
all infinitely differentiable functions on $\mathbb R^n$ with
compact support.
It is known that, if the even part of $\Omega$ belongs to
$L(\log L)(\mathbb S^{n-1})$, then $T$ extends to a bounded operator
on $L^p(\mathbb R^n)$ for $p\in(1,\infty)$; see \cite{cz56}. If
further assume that $\Omega\in L(\log L)(\mathbb S^{n-1})$,
then Seeger \cite{s96} proved that $T$ extends to an operator
of weak type $(1,1)$.

For simplicity, let $\mathcal L:=L(\log L)(\mathbb S^{n-1})$ and we denote by
$\|\cdot\|_{\mathcal L}$ its \emph{Luxemburg norm}; that is, with
$\Phi(s):=s(1+\log^+s)$ for $s\in(0,\infty)$,
$$\|h\|_{\mathcal L}:=\inf\left\{A\in(0,\infty):
\fint_{\mathbb S^{n-1}}\Phi\left(\frac{|h(\theta)|}{A}\right)\,d\sigma(\theta)
\le1\right\}.$$
For any $\delta\in (0,1)$, define the \emph{translation modulus}
\begin{align*}
\omega_{\mathcal L}(\Omega,\delta)
:=
\sup_{h\in\mathbb R^n,\, |h|\le\delta}
\|\Omega(\cdot+h)-\Omega(\cdot)\|_{\mathcal L},
\end{align*}
where $\Omega$ is understood as a homogeneous function of degree zero on
$\mathbb R^n\setminus\{{\bf0}\}$ with ${\bf 0}$ denoting
the origin of $\mathbb R^n$.
For any $\varrho\in(0,1)$, we say that $\Omega$ satisfies
the \emph{$\varrho$-$L(\log L)$-Dini condition} if
\begin{align*}
\mathfrak D_{\mathcal L,\varrho}(\Omega)
:=
\|\Omega\|_{\mathcal L}
+
\left[
\int_0^{1}\omega_{\mathcal L}(\Omega,\delta)^\varrho
\frac{d\delta}{\delta}
\right]^{1/\varrho}
<\infty.
\end{align*}
\begin{remark}
We note that the $L^1$-Dini condition is sometimes defined in the literature
via a rotational modulus of continuity, which is defined as
$\mathfrak{D}_{\mathcal{L},1}$ with
$\mathcal{L} = L^1(\mathbb S^{n-1})$ and $\omega_{\mathcal L}(\Omega,\delta)$ replaced by
$$\sup_{\rho \in {\rm SO}(n), \|\rho - I\| \le \delta} \|\Omega(\rho\cdot) -
\Omega(\cdot)\|_{L^1(\mathbb S^{n-1})},$$
where $\|\cdot\|$ denotes the matrix norm
and ${\rm SO}(n)$
denotes the special orthogonal group.
Since it is well known that
this rotational definition is equivalent to the translational
one in the $L^1$-case (see \cite[Theorem 2.5]{dl17-2}),
it is an interesting question whether this equivalence still holds for the
$\varrho$-$L(\log L)$-Dini condition.
We do not address this question here, as it does not
affect our present results.
\end{remark}
In what follows, we always fix
$\varrho\in(0,1)$ such that $\|\cdot\|_{L^{1,\infty}(\mathbb R^{n+1}_+,\nu_\gamma)}^\varrho$
satisfies the triangle inequality and simply say that $\Omega$
satisfies the $L(\log L)$-Dini condition, denoted
simply by $\mathfrak{D}_{\mathcal L}(\Omega)<\infty$.
The existence of such $\varrho$
is due to the Aoki--Rolewicz theorem.

For any $k\in\mathbb Z_+$,  the
\emph{angular oscillation of $\Omega$ at scale $2^{-k}$} is defined by
setting, for any $\xi\in\mathbb S^{n-1},$
\begin{align*}
\Omega_k^\sharp(\xi)
:=
2^{k(n-1)}
\int_{\{\eta\in\mathbb S^{n-1}:|\eta-\xi|<2^{-k}\}}
|\Omega(\xi)-\Omega(\eta)|\,d\sigma(\eta).
\end{align*}

\begin{lemma}
\label{lem-implies-angular-sum}
Assume that $\Omega$ satisfies the $L(\log L)$-Dini condition.
Then
\[
\left[
\sum_{k=0}^\infty
\left\|\Omega_k^\sharp\right\|_{\mathcal L}^{\varrho}
\right]^{\frac1\varrho}
\lesssim
\mathfrak D_{\mathcal L}(\Omega).
\]
\end{lemma}

\begin{proof}
Note that, for any $\delta\in(0,1]$, any
$\xi\in\mathbb S^{n-1}$, and any nonnegative
measurable functions $H$ on $\mathbb S^{n-1}$,
\begin{equation}\label{eq5.6}
\delta^{-(n-1)}
\int_{\{\eta\in\mathbb S^{n-1}:|\eta-\xi|<\delta\}}
H(\eta)\,d\sigma(\eta)
\lesssim
\fint_{\{h\in\mathbb R^n:|h|<2\delta\}}
H\left(\frac{\xi+h}{|\xi+h|}\right)\,dh,
\end{equation}
where the implicit positive constant depends only on $n$.
Indeed, for any $\xi\in\mathbb{S}^{n-1}$, using the change of variables $z=\xi+h$
and then using polar coordinates $z=r\eta$,
we find that
\[
\begin{aligned}
\int_{|h|<2\delta}
H\left(\frac{\xi+h}{|\xi+h|}\right)\,dh
&=
\int_{|z-\xi|<2\delta}H(z')\,dz\\
&=
\int_{\mathbb S^{n-1}}H(\eta)
\int_0^\infty
\mathbf 1_{\{|r\eta-\xi|<2\delta\}}
r^{n-1}\,dr\,d\sigma(\eta).
\end{aligned}
\]
If $|\eta-\xi|<\delta$ and $1<r<1+\delta$, then
\[
|r\eta-\xi|
\le |r\eta-\eta|+|\eta-\xi|
\le 2\delta.
\]
Thus, for any $\xi\in\mathbb{S}^{n-1}$,
\[
\int_{|h|<2\delta}
H\left(\frac{\xi+h}{|\xi+h|}\right)\,dh
\ge
\delta
\int_{|\eta-\xi|<\delta}H(\eta)\,d\sigma(\eta).
\]
Then dividing by $\delta^n$
in both sides proves \eqref{eq5.6}.

Now, applying \eqref{eq5.6} with $\delta=2^{-k}$ for any $k\in\mathbb{Z}_+$ and
\(
H(\eta):=|\Omega(\xi)-\Omega(\eta)|,
\)
we conclude that, for any $\xi\in\mathbb S^{n-1}$,
\begin{align*}
\Omega_k^\sharp(\xi)
\lesssim
\fint_{|h|<2^{-k+1}}
|\Omega(\xi)-\Omega(\xi+h)|\,dh.
\end{align*}
Here $\Omega$ is understood as a homogeneous function of degree zero on
$\mathbb R^n\setminus\{{\bf 0}\}$.
Taking the Orlicz norm $\mathcal{L}$ and using Minkowski's integral inequality
in $\mathcal L$ (see, for instance, \cite[Proposition 2.1]{s95}), we obtain, for any
$k\in\mathbb{N}$,
\begin{align*}
\left\|\Omega_k^\sharp\right\|_{\mathcal L}
\lesssim
\fint_{|h|<2^{-k+1}}
\|\Omega(\cdot)-\Omega(\cdot+h)\|_{\mathcal L}\,dh\le
\omega_{\mathcal L}(\Omega,2^{-k+1}).
\end{align*}
For $k=0$, we simply use
\[
\Omega_0^\sharp(\xi)
\le
|\Omega(\xi)|
+
\int_{\mathbb S^{n-1}}|\Omega(\eta)|\,d\sigma(\eta),
\]
which further implies that
\(
\|\Omega_0^\sharp\|_{\mathcal L}
\lesssim
\|\Omega\|_{\mathcal L}.
\)
Since $\omega_{\mathcal L}(\Omega,\delta)$ is nondecreasing in $\delta$,
it follows that
\begin{align*}
\sum_{k=0}^\infty
\left\|\Omega_k^\sharp\right\|_{\mathcal L}^{\varrho}
&\lesssim
\|\Omega\|_{\mathcal L}^{\varrho}
+
\sum_{k\in\mathbb{N}}
\omega_{\mathcal L}(\Omega,2^{-k+1})^\varrho             \\
&\lesssim
\|\Omega\|_{\mathcal L}^{\varrho}
+
\int_0^{1}
\omega_{\mathcal L}(\Omega,\delta)^\varrho
\frac{d\delta}{\delta}.
\end{align*}
Taking the power $\frac1\varrho$
in both sides yields the desired estimate.
This finishes the proof of Lemma \ref{lem-implies-angular-sum}.
\end{proof}

We now prove the following Cotlar-type inequality
with respect to truncated rough singular integrals.

\begin{proposition}\label{prop-rough-Cotlar}
Let $\Omega\in L(\log L)(\mathbb{S}^{n-1})$ satisfy \eqref{5.1}.
Then, for any $f\in \bigcup_{p\in[1,\infty)}L^p(\mathbb{R}^{n})$,
any $t\in(0,\infty)$, and almost every $x\in\mathbb R^n$,
\begin{align*}
|T_{\Omega,t}(f)(x)|
&\lesssim
\mathcal{M}_{1}(T_{\Omega}(f))(x,t)
+
\mathcal{M}_1(f)(x,t)\\
&\quad+
\sum_{k=0}^\infty 2^{-k}\mathcal{M}_{2^{k+1}}^{\Omega}(f)(x,t)
+
\sum_{k=0}^\infty \mathcal{M}_{2^{k+1}}^{\Omega_k^\sharp}(f)(x,t),
\end{align*}
where $\mathcal{M}_1$ is as in \eqref{eq-lifted-hl-def}
and the implicit positive constant is independent of $f,t$, and $x$.
\end{proposition}

\begin{proof}
Fix $x\in\mathbb R^n$, $t\in(0,\infty)$, and $f\in \bigcup_{p\in[1,\infty)}L^p(\mathbb{R}^{n})$.
Let
$B:=B(x,\frac{t}{2})$, $f_0:=f\mathbf 1_{B(x,t)},$ and
$f_\infty:=f-f_0.$
For almost every $z\in B$, since $\operatorname{supp}f_\infty\subset
B(x,t)^\complement$, it follows that
\begin{align*}
T_{\Omega}(f_\infty)(z)&=\int_{\mathbb R^n}K_\Omega(z-y)f_{\infty}(y)\,dy
=\int_{|x-y|>t}K_\Omega(z-y)f(y)\,dy   \\
&=T_{\Omega,t}(f)(x) - \int_{|x-y|>t}
[K_\Omega(x-y)-K_\Omega(z-y)]f(y)\,dy .
\end{align*}
Let $q\in(0,1)$. Taking the $L^q$ average over $B$ with respect to $z$
gives
\begin{align}\label{eq5.10}
|T_{\Omega,t}(f)(x)|&
\lesssim
\left[\fint_B |T_{\Omega}(f)(z)|^q\,dz\right]^{\frac{1}{q}}
+
\left[\fint_B |T_{\Omega}(f_0)(z)|^q\,dz\right]^{\frac{1}{q}}  \nonumber\\
&\quad+
\fint_B
\int_{|x-y|>t}
|K_\Omega(x-y)-K_\Omega(z-y)|\,|f(y)|\,dy\,dz .
\end{align}
Since $q<1$, we deduce that
\begin{align}\label{eq5.9}
\left[\fint_B |T_{\Omega}(f)(z)|^q\,dz\right]^{\frac{1}{q}}
\lesssim
\fint_{B(x,t)} |T_{\Omega}(f)(z)|\,dz
\le
\mathcal{M}_{1}(T_{\Omega}(f))(x,t).
\end{align}
Moreover, by the weak type $(1,1)$ boundedness of $T_\Omega$
(see \cite[Proposition]{s96}) and Kolmogorov's
inequality, we find that
\begin{align}\label{eq5.8}
\left[\fint_B |T_{\Omega}(f_0)(z)|^q\,dz\right]^{\frac{1}{q}}
\lesssim
|B|^{-1}\|f_0\|_{L^1(\mathbb R^n)}
\lesssim
\mathcal{M}_1(f)(x,t).
\end{align}

It remains to estimate the last term of \eqref{eq5.10}. Write
\(u:=x-y\) and \(h:=z-x.\)
Then $|h|<\frac{t}{2}$ and $|u|>t$. Given $k\in\mathbb{Z}_+$, define
\( A_k:=\{u\in\mathbb R^n:2^kt<|u|\le2^{k+1}t\}.
\)
When $u\in A_k$ and $|h|<\frac{t}{2}$,
we have $|u+h|\sim |u|\sim2^kt$ and hence
\begin{align}\label{eq-kk}
|K_\Omega(u)-K_\Omega(u+h)|
\le        |\Omega(u')|
\left|\frac1{|u|^n}-\frac1{|u+h|^n}\right|+
\frac{|\Omega(u')-\Omega((u+h)')|}{|u+h|^n},
\end{align}
where $u':=\frac{u}{|u|}$. The radial part satisfies
\begin{align}\label{eq5.14}
\left|\frac1{|u|^n}-\frac1{|u+h|^n}\right|
\lesssim
\frac{|h|}{|u|^{n+1}}
\lesssim
2^{-k}(2^kt)^{-n}.
\end{align}
For the angular part, applying polar coordinates and an argument
similar to that used in the proof of \eqref{eq-hu}, we conclude that, for any $u\in A_k$,
\begin{align*}
\fint_{|h|<\frac{t}{2}}
|\Omega(u')-\Omega((u+h)')|\,dh
\lesssim
\Omega_k^\sharp(u').
\end{align*}
Combining this, \eqref{eq-kk}, \eqref{eq5.14},  and
Tonelli's theorem, we obtain
\begin{align*}
&\fint_B
\int_{|x-y|>t}
|K_\Omega(x-y)-K_\Omega(z-y)|\,|f(y)|\,dy\,dz \\
&\quad\lesssim
\sum_{k=0}^\infty
2^{-k}
\frac1{(2^kt)^n}
\int_{B(x,2^{k+1}t)}
|\Omega((x-y)')|\,|f(y)|\,dy  \\
&\qquad+
\sum_{k=0}^\infty
\frac1{(2^kt)^n}
\int_{B(x,2^{k+1}t)}
\Omega_k^\sharp((x-y)')\,|f(y)|\,dy \\
&\quad\lesssim
\sum_{k=0}^\infty 2^{-k}\mathcal{M}_{2^{k+1}}^{\Omega}(f)(x,t)
+
\sum_{k=0}^\infty \mathcal{M}_{2^{k+1}}^{\Omega_k^\sharp}(f)(x,t).
\end{align*}
Combining this, \eqref{eq5.8}, \eqref{eq5.9},
and \eqref{eq5.10}, we obtain the desired conclusion,
which completes the proof of Proposition \ref{prop-rough-Cotlar}.
\end{proof}

We now state the weak-type representation of truncated rough singular integrals.
\begin{theorem}
\label{thm-truncated-rough-representation}
Let $\gamma\in(-\infty,-n)\cup(0,\infty)$, and let $\Omega$ satisfy
\eqref{5.1} and the $L(\log L)$-Dini condition.
Then there exists a positive constant $C_1$, depending only on $n$
and $\gamma$, such that, for any $f\in L^1(\mathbb R^n)$ with $T_{\Omega}(f)\in L^1(\mathbb R^n)$,
\begin{align}\label{eq5.17}
\sup_{\lambda\in(0,\infty)}
\lambda
\underset{|T_{\Omega,t}(f)(x,t)|>\lambda t^\gamma}
{\int_{\mathbb R^n}\int_0^\infty}
t^{\gamma-1}\,dt\,dx
\le C_1\left\{
\|T_{\Omega}(f)\|_{L^1(\mathbb R^n)}
+
\left[\mathfrak D_{\mathcal L}(\Omega)+1\right]
\|f\|_{L^1(\mathbb R^n)}\right\}.
\end{align}

In addition, if $T_\Omega$ satisfies that, for any $f\in C_{\rm c}^\infty(\mathbb{R}^n)$,
\begin{equation}\label{eq-pw}
        \lim_{t\to0^+}T_{\Omega,t}(f)(x)=T_{\Omega}(f)(x)
        \quad\text{for a.e. }x\in\mathbb R^n.
\end{equation}
Then there exists a positive constant $C_2$ such that, for any $f\in L^1(\mathbb{R}^n)$,
\begin{align}\label{eq-lower-T}
C_2\|T_{\Omega}(f)\|_{L^1(\mathbb R^n)}\le
\sup_{\lambda\in(0,\infty)}
\lambda
\underset{|T_{\Omega,t}(f)(x,t)|>\lambda t^\gamma}
{\int_{\mathbb R^n}\int_0^\infty}
t^{\gamma-1}\,dt\,dx.
\end{align}
\end{theorem}

\begin{proof}
We first prove the upper estimate \eqref{eq5.17}. By Proposition~\ref{prop-rough-Cotlar}
and the Aoki--Rolewicz theorem, we find that
\begin{align*}
\left\|\frac{T_{\Omega,t}(f)(x)}{t^\gamma} \right\|_{L^{1,\infty}(\mathbb{R}^{n+1}_+,
\nu_\gamma)}^\varrho
&\le\left\|\frac{\mathcal{M}_{1}(T_{\Omega}(f))(x,t)}{t^\gamma}
\right\|_{L^{1,\infty}(\mathbb{R}^{n+1}_+,\nu_\gamma)}^\varrho
+\left\|\frac{\mathcal{M}_1(f)(x,t)}{t^\gamma}
\right\|_{L^{1,\infty}(\mathbb{R}^{n+1}_+,\nu_\gamma)}^\varrho\\
&\quad+\sum_{k=0}^\infty 2^{-k\varrho}\left\|\frac{\mathcal{M}_{2^{k+1}}^{\Omega}(f)(x,t)}{t^\gamma}
\right\|_{L^{1,\infty}(\mathbb{R}^{n+1}_+,\nu_\gamma)}^\varrho\\
&\quad+\sum_{k=0}^\infty
\left\|\frac{\mathcal{M}_{2^{k+1}}^{\Omega_k^\sharp}(f)(x,t)}{t^\gamma}
\right\|_{L^{1,\infty}(\mathbb{R}^{n+1}_+,\nu_\gamma)}^\varrho\\
&=:{\rm I}+{\rm II}+{\rm III}+{\rm IV},
\end{align*}
where $\nu_\gamma$ is as in \eqref{dv}.
From \eqref{eq-lifted-hl-weak}, it follows that
${\rm I}\lesssim \left\|T_{\Omega}f \right\|_{L^1(\mathbb{R}^n)}^\varrho$ and
${\rm II}\lesssim \left\|f \right\|_{L^1(\mathbb{R}^n)}^\varrho.$
Applying Theorem \ref{thm-diag}(ii), we conclude that
\[
{\rm III}\lesssim
\sum_{k=0}^\infty 2^{-k\varrho}\|f\|_{L^1(\mathbb{R}^n)}^\varrho\sim
\|f\|_{L^1(\mathbb{R}^n)}^{\varrho}.
\]
It remains to estimate ${\rm IV}$. Indeed, applying
Theorem \ref{thm-up-p=1} with $\Omega$ therein replaced by $\frac{\Omega}{\|\Omega\|_{\mathcal{L}}}$
and using $|\frac{\Omega}{\|\Omega\|_{\mathcal{L}}}|_{L(\log L)(\mathbb S^{n-1})}\le 1$, we obtain
\begin{equation*}
\left\|\frac{\mathcal{M}_{\theta}^\Omega(f)(x,t)}{t^\gamma}
\right\|_{L^{1,\infty}(\mathbb{R}^{n+1}_+,\nu_\gamma)}\lesssim
\|\Omega\|_{\mathcal L}\|f\|_{L^1(\mathbb R^n)},
\end{equation*}
where the implicit positive constant depends only on $n$ and $\gamma$.
By this
and Lemma \ref{lem-implies-angular-sum}, we find that
\begin{align*}
{\rm IV}\lesssim \sum_{k=0}^\infty
\left\|\Omega_k^\sharp \right\|_{\mathcal{L}}^{\varrho}\|f\|_{L^1(\mathbb{R}^n)}^{\varrho}
\lesssim\mathfrak D_{\mathcal L}(\Omega)^\varrho
\|f\|_{L^1(\mathbb R^n)}^\varrho.
\end{align*}
Consequently,
\[
\left\|\frac{T_{\Omega,t}(f)(x)}{t^\gamma} \right\|_{L^{1,\infty}(\mathbb{R}^{n+1}_+,
\nu_\gamma)}
\lesssim\left[\mathfrak D_{\mathcal L}(\Omega)+1\right]\|f\|_{L^1(\mathbb R^n)}+
\left\|T_{\Omega}f \right\|_{L^1(\mathbb{R}^n)}.
\]
This proves the upper estimate \eqref{eq5.17}.

It remains to show the lower estimate \eqref{eq-lower-T}.
By the assumption \eqref{eq-pw} and \cite[Theorem 1.1]{l25},
we easily find that \eqref{eq-pw} holds for all $f\in L^1(\mathbb{R}^n)$.
From this and \cite[Lemma 3.3]{dlyyz26}, we deduce the lower estimate \eqref{eq-lower-T}.
This finishes the proof of Theorem \ref{thm-truncated-rough-representation}.
\end{proof}

In what follows, fix $\Phi\in\mathcal{S}(\mathbb{R}^n)$
satisfying $\int_{\mathbb{R}^n}\Phi(x)\,dx\ne0$.
The \emph{Hardy space $H^1(\mathbb{R}^n)$} consists of all $f\in L^1(\mathbb{R}^n)$
such that
\begin{equation*}
\|f\|_{H^1(\mathbb{R}^n)}:=\left\|\sup_{t\in(0,\infty)}
\left|\Phi_t\ast f\right|\right\|_{L^1(\mathbb{R}^n)}<\infty,
\end{equation*}
where $\Phi_t(\cdot):=\frac{1}{t^{n}}\Phi(\frac{\cdot}{t})$ for any
$t\in(0,\infty)$.

A classical result establishes that $H^1(\mathbb{R}^n)$
admits an equivalent characterization through singular integrals,
specifically the Riesz transforms $\{R_j\}_{j=1}^n$
(see, for example, \cite{fs72,u82}).
Indeed, $f\in H^1(\mathbb{R}^n)$ if and only if $f\in L^1(\mathbb{R}^n)$
and $R_j(f)\in L^1(\mathbb{R}^n)$ for all $j\in\{1,\ldots,n\}$.
Moreover, we have the norm equivalence
\begin{equation}\label{5-9-0}
\|f\|_{H^1(\mathbb{R}^n)}
\sim \|f\|_{L^1(\mathbb{R}^n)} +\sum_{j=1}^n \|R_j(f)\|_{L^1(\mathbb{R}^n)},
\end{equation}
where the positive equivalence constants are independent
of $f$.
Recall that the \emph{$j$-th Riesz transform} is defined pointwise
for almost every $x\in\mathbb{R}^n$ by the limit
$R_j(f)(x):=\lim_{t\to0^+} R_j^t(f)(x)$,
where $R_j^t$ denotes the truncated Riesz transform,
given by
\begin{equation*}
R_{j}^t(f)(x):=\int_{|y|>t}f(x-y) \frac{y_j}{|y|^{n+1}}\,dy.
\end{equation*}

Motivated by this,
we obtain the following characterization of Hardy spaces.
Since its proof is a direct consequence of Theorem \ref{thm-truncated-rough-representation},
we omit the details.

\begin{corollary}
\label{cor-H1-rough-trunc}
Let $\gamma\in(-\infty,-n)\cup(0,\infty)$ and $N\in\mathbb{N}$. For each
$\ell\in\{1,\ldots,N\}$, let $\Omega^{(\ell)}$ satisfy the assumptions of
Theorem~\ref{thm-truncated-rough-representation}. Assume further that,
for any $f\in L^1(\mathbb R^n)$,
\begin{align}\label{eq5.21}
\|f\|_{H^1(\mathbb R^n)}
\sim
\|f\|_{L^1(\mathbb R^n)}
+
\sum_{\ell=1}^N
\|T_{\Omega^{(\ell)}}f\|_{L^1(\mathbb R^n)}.
\end{align}
Then
\begin{align*}
\|f\|_{H^1(\mathbb R^n)}
\sim
\|f\|_{L^1(\mathbb R^n)}
+
\sum_{\ell=1}^N
\left\|
\frac{T_{\Omega^{(\ell)},t}f(x)}{t^\gamma}
\right\|_{L^{1,\infty}(\mathbb R^{n+1}_+,\nu_\gamma)}.
\end{align*}
Here the positive equivalence constants are independent of $f$.
\end{corollary}

\begin{remark}
In the classical case of Riesz transforms, the
assumption \eqref{eq5.21} is precisely
the Fefferman--Stein characterization of $H^1(\mathbb R^n)$
[see \cite{fs72} and see also \eqref{5-9-0}] and, in this case, Corollary \ref{cor-H1-rough-trunc}
coincides with \cite[Theorem 5.1]{dlyyz26}.

For more general case, let $\theta_1,\ldots,\theta_m\in
C^\infty(\mathbb{S}^{n-1})$. It is known that,
for each $j\in\{1,\ldots,m\}$, there exist
$\alpha_{\theta_j}\in\mathbb{C}$ and $\Omega_{\theta_j}\in
C^\infty(\mathbb{S}^{n-1})$ satisfying \eqref{5.1} such that
the Fourier multiplier operator $K_jf = (\theta_j(\frac{\xi}{|\xi|})\widehat{f})^\vee$
satisfies that, for any $f\in \bigcup_{p\in[1,\infty)}L^p(\mathbb{R}^{n})$
and almost every $x\in\mathbb{R}^n$,
\begin{align}\label{eq-u1}
K_jf(x) = \alpha_{\theta_j}f(x)+ T_{\Omega_{\theta_j}}f(x);
\end{align}
see \cite[p.\,75]{s70} and also \cite[Lemma 2.A]{u82}.
Uchiyama \cite[Corollary 1]{u82} proved that, if
\begin{align*}
\operatorname{rank}
\begin{pmatrix}
\theta_1(\xi) & \ldots & \theta_m(\xi) \\
\theta_1(-\xi) & \ldots & \theta_m(-\xi)
\end{pmatrix}=2
\end{align*}on $\mathbb{S}^{n-1}$,
then, for any $f\in L^1(\mathbb{R}^n)$,
\begin{align}\label{eq-u2}
\|f\|_{H^1(\mathbb{R}^n)}\sim \sum_{j=1}^m\left\|
K_jf \right\|_{L^1(\mathbb{R}^n)}.
\end{align}
In this case, one can verify that
each $\Omega_{\theta_j}$,
$j\in\{1,\ldots,m\}$, satisfies both the
$L(\log L)$-Dini condition
and \eqref{eq-pw}.
Moreover, \eqref{eq-u1} and \eqref{eq-u2}
ensure that the assumption \eqref{eq5.21} with $N:=m$
and $\Omega^{(\ell)} := \Omega_{\theta_{\ell}}$ holds.
Thus, in this case, Corollary \ref{cor-H1-rough-trunc}
yields a new characterization of $H^1(\mathbb{R}^n)$.
\end{remark}

\smallskip
\noindent\textbf{Acknowledgements}\quad
The authors would like to sincerely thank Professor Guoen Hu for many
insightful and illuminating discussions
on rough singular integrals and the characterization of Hardy spaces.

\bigskip

\noindent
Dachun Yang (Corresponding author),
Wen Yuan and Yirui Zhao

\medskip

\noindent  Laboratory of Mathematics and Complex Systems
(Ministry of Education of China),
School of Mathematical Sciences, Institute for Advanced Study,  Beijing Normal University,
Beijing 100875, The People's Republic of China

\smallskip

\noindent{\it E-mails:} \texttt{dcyang@bnu.edu.cn} (D. Yang)

\noindent\phantom{\it E-mails:} \texttt{wenyuan@bnu.edu.cn} (W. Yuan)

\noindent\phantom{\it E-mails:} \texttt{yiruizhao@mail.bnu.edu.cn} (Y. Zhao)

\end{document}